\documentclass[a4paper,11pt]{amsart}

\usepackage{amscd,amssymb}
\usepackage{mathrsfs}
\usepackage{bm}
\usepackage{color}
\usepackage{multirow,bigdelim}
\usepackage[all]{xy}
\usepackage{hyperref}
\usepackage{tikz}

\usepackage{geometry}
\newtheorem{thm}{Theorem}[section]
\newtheorem{lem}[thm]{Lemma}
\newtheorem{cor}[thm]{Corollary}
\newtheorem{prop}[thm]{Proposition}

\newtheorem{conj}[thm]{Conjecture}

\theoremstyle{definition}
\newtheorem{defin}[thm]{Definition}

\newtheorem*{organization}{Organization of article}
\newtheorem*{conventions}{Conventions}
\newtheorem*{notation}{Notation}
\newtheorem*{acknowledgment}{Acknowledgment}

\theoremstyle{remark}
\newtheorem{rem}[thm]{Remark}

\numberwithin{equation}{section}

\newcommand{\bA}{\mathbb{A}}

\newcommand{\bF}{\mathbb{F}}

\newcommand{\bG}{\mathbb{G}}
\newcommand{\bk}{\Bbbk}

\newcommand{\sO}{\mathscr{O}}
\newcommand{\bP}{\mathbb{P}}

\newcommand{\bQ}{\mathbb{Q}}

\newcommand{\bZ}{\mathbb{Z}}
\newcommand{\Amp}{\mathrm{Amp}}
\newcommand{\Ampc}{\mathrm{Amp}^{\mathrm{cyl}}}

\newcommand{\Cl}{\mathrm{Cl}}

\newcommand{\Pic}{\mathrm{Pic}}

\newcommand{\Sing}{\mathrm{Sing}}

\newcommand{\Spec}{\mathrm{Spec}}
\newcommand{\Supp}{\mathrm{Supp}}

\newcommand{\wA}{\widetilde{A}}
\newcommand{\wB}{\widetilde{B}}
\newcommand{\wC}{\widetilde{C}}
\newcommand{\wD}{\widetilde{D}}
\newcommand{\wE}{\widetilde{E}}
\newcommand{\wF}{\widetilde{F}}

\newcommand{\wS}{\widetilde{S}}
\newcommand{\wGamma}{\widetilde{\Gamma}}
\newcommand{\wDelta}{\widetilde{\Delta}}

\newcommand{\hF}{\hat{F}}
\newcommand{\hM}{\hat{M}}

\newcommand{\hGamma}{\hat{\Gamma}}

\newcommand{\sA}{\textsf{A}}
\newcommand{\sD}{\textsf{D}}
\newcommand{\sE}{\textsf{E}}

\begin{document}

\title{Polarized cylinders in Du Val del Pezzo surfaces of degree two}


\author{}
\address{}
\curraddr{}
\email{}
\thanks{}

\author{Masatomo Sawahara}
\address{Faculty of Education, Hirosaki University, Bunkyocho 1, Hirosaki-shi, Aomori 036-8560, JAPAN}
\curraddr{}
\email{sawahara.masatomo@gmail.com}
\thanks{}



\date{}

\dedicatory{}

\begin{abstract}
Let $S$ be a del Pezzo surface with at worst Du Val singularities of degree $2$ such that $S$ admits an $(-K_S)$-polar cylinder. 
In this article, we construct an $H$-polar cylinder for any ample $\mathbb{Q}$-divisor $H$ on $S$. 
\end{abstract}

\maketitle
\setcounter{tocdepth}{1}

Throughout this article, all considered varieties are assumed to be algebraic and defined over an algebraically closed field $\bk$ of characteristic $0$. 
\section{Introduction}\label{1}
Let $X$ be a normal projective variety. 
An open subset $U$ in $X$ is called a {\it cylinder} if $U$ is isomorphic to $\bA ^1_{\bk} \times Z$ for some variety $Z$. 
Moreover, letting $H$ be an ample $\bQ$-divisor on $X$, we say that a cylinder $U$ in $X$ is an {\it $H$-polar cylinder} if there exists an effective $\bQ$-divisor $D$ on $X$ such that $D \sim _{\bQ} H$ and $U = X \backslash \Supp (D)$. 
The motivation for studying polarized cylinders comes from the study of $\bG _a$-actions on the corresponding affine cones, which are special kinds of affine varieties. 
\begin{thm}[{\cite{KPZ11,KPZ13}}]
Let $X$ be a normal projective variety, and let $H$ be an ample $\bQ$-divisor on $X$. 
Then the following affine variety: 
\begin{align*}
\Spec \left( \bigoplus _{n \ge 0} H^0(X,\sO _X(nH)) \right)
\end{align*}
admits a non-trivial $\bG _a$-action if and only if $X$ contains an $H$-polar cylinder. 
\end{thm}
Moreover, the geometry of polarized cylinders in projective varieties can be applied to the study of the flexibility of affine cones and 
the existence of $\bG _a$-actions on the complements of hypersurfaces (see {\cite{Pre13,PW16,CDP18,Par22}}). 
Meanwhile, the connection between anti-canonical polarized cylinders in Fano varieties and their $\alpha$-invariants is also known (see {\cite[Theorem 1.26]{CPPZ21}}). 

In this article, we study polarized cylinders in log del Pezzo surfaces with the aim of making explicit several geometric and group-theoretic applications. 
In this article, we study polarized cylinders in log del Pezzo surfaces in order to make explicit progress. 
Here, a log del Pezzo surface means a del Pezzo surface (see \S \S \ref{2-2}, for the definition) with at worst klt singularities. 
We introduce the following notation: 
\begin{defin}
Let $S$ be a normal projective rational surface. Then we say that the following set is called the {\it cylindrical ample set} of $S$: 
\begin{align*}
\Ampc (S) := \{ H \in \Amp (S) \,|\, \text{there is an $H$-polar cylinder on $S$}\}, 
\end{align*}
where $\Amp (S)$ is the ample cone of $S$. 
\end{defin}
Cheltsov, Park and Won proposed the following conjecture on polarized cylinders in log del Pezzo surfaces: 
\begin{conj}[{\cite{CPW17}}]\label{conj}
Let $S$ be a log del Pezzo surface. 
Then $-K_S \in \Ampc (S)$ if and only if $\Ampc (S) = \Amp (S)$. 
\end{conj}
Note that Conjecture \ref{conj} holds for smooth del Pezzo surfaces because it follows from the following two theorems: 
\begin{thm}[{\cite{KPZ11,KPZ14,CPW16a}}]
Let $S$ be a smooth del Pezzo surface of degree $d \in \{ 1,\dots ,9\}$ (see \S \S \ref{2-2}, for the definition). 
Then $-K_S \in \Ampc (S)$ if and only if $d \ge 4$. 
\end{thm}
\begin{thm}[{\cite{CPW17}}]
Let $S$ be a smooth del Pezzo surface of degree $\ge 4$. 
Then $\Ampc (S) = \Amp (S)$. 
\end{thm}
\begin{rem}
{\cite{CPW17,KP21,KW25}} further study the structure of cylindrical ample sets of smooth del Pezzo surfaces of degree $\le 3$. 
See also {\cite{MW18,Che21}} for the structure of cylindrical ample sets of smooth rational surfaces. 
\end{rem}

From now on, we consider polarized cylinders in del Pezzo surfaces with at worst Du Val singularities; i.e.,  Du Val del Pezzo surfaces. 
Notice that every Du Val del Pezzo surface is a log del Pezzo surface because all Du Val singularities are klt singularities. 
The condition for a Du Val del Pezzo surface to contain anti-canonical polar cylinders is completely determined. 
More precisely, the following theorem holds: 
\begin{thm}[{\cite{CPW16b}}]\label{CPW}
Let $S$ be a Du Val del Pezzo surface of degree $d \in \{ 1,\dots ,9\}$. Then: 
\begin{enumerate}
\item If $d \ge 4$, then $-K_S \in \Ampc (S)$. 
\item If $d=3$, then $-K_S \in \Ampc (S)$ if and only if $S$ has a singular point. 
\item If $d=2$, then $-K_S \in \Ampc (S)$ if and only if $S$ has a singular point, which is not of type $\sA _1$. 
\item If $d=1$, then $-K_S \in \Ampc (S)$ if and only if $S$ has a singular point, which is not of type $\sA _1$, $\sA _2$, $\sA _3$ nor $\sD _4$. 
\end{enumerate}
\end{thm}
\begin{rem}
For a Du Val del Pezzo surface $S$, Belousov shows that $\Ampc (S) \not= \emptyset$ if and only if ${\rm Dyn}(S) \not= 4\sA _2$, $2\sA _3 + 2\sA _1$ or $2\sD _4$ (see {\cite{Bel17,Bel23}}). 
\end{rem}
\begin{rem}
There are infinitely many log del Pezzo surfaces without anti-canonical polar cylinders (see {\cite{KKW24,CD,KKW}}). 
However, the condition for whether a (general) log del Pezzo surface contains anti-canonical polar cylinders is still unknown. 
\end{rem}
In {\cite{Saw25}}, the author studies polarized cylinders in Du Val del Pezzo surfaces of degree $\ge 3$. 
In this article, we construct numerous polarized cylinders on singular rational surfaces that were not covered in the previous work. 
As a result, we completely determine the structure of cylindrical ample sets of Du Val del Pezzo surfaces of degree $2$ with anti-canonical polar cylinders. 
Our main result is summarized as follows: 
\begin{thm}\label{main(1)}
Let $S$ be a Du Val del Pezzo surface of degree $\ge 2$. 
Then $-K_S \in \Ampc (S)$ if and only if $\Ampc (S) = \Amp (S)$. 
\end{thm}
Theorem \ref{main(1)} provides a partially positive answer to Conjecture \ref{conj} as follows: 
\begin{cor}
Conjecture \ref{conj} is true for Du Val del Pezzo surfaces with degree $\ge 2$.
\end{cor}
\begin{organization}
In Section \ref{2}, we prepare results on cylinders in Hirzebruch surfaces and singular fibers of $\bP ^1$-fibrations. 
Moreover, we review the relationship between Du Val del Pezzo surfaces and weak del Pezzo surfaces, and then discuss the singularity types of Du Val del Pezzo surfaces. 

In Section \ref{3}, we find a special kind of $(-1)$-curves on Du Val del Pezzo surfaces of degree $2$ with a singular point of type $\sA _n$ $(n \ge 3)$ referring to {\cite[\S 4]{Saw24}}. 
Using this result, we can obtain $\bP ^1$-fibrations from these surfaces, which can construct polarized cylinders. 

In Section \ref{4}, we present a method to construct $H$-polar cylinders for every ample $\bQ$-divisor $H$ on normal rational surfaces with specific $\bP ^1$-fibration structures. 
Our argument is similar to {\cite[\S 3]{Saw25}} but requires a more intricate discussion. 

In the final Section \ref{5}, we prove Theorem \ref{main(1)}, which is our main result. 
In more detail, using results in Section \ref{3}, we will show that every Du Val del Pezzo surface of degree $2$ with anti-canonical polar cylinders has a specific $\bP ^1$-fibration in the above sense.
Namely, Theorem \ref{main(1)} is shown by this argument combined with the result of Section \ref{4}. 
\end{organization}
\begin{conventions}
For an integer $m$, we say that an $m$-curve on a smooth projective surface is called a smooth projective rational curve with self-intersection number $m$. 

For any weighted dual graph, a vertex $\circ$ with the number $m$ corresponds to an $m$-curve. 
Exceptionally, we omit this weight (resp. we omit this weight and use the vertex $\bullet$ instead of $\circ$) if $m=-2$ (resp. $m=-1$).
\end{conventions}
\begin{notation}
We employ the following notations: 
\begin{itemize}
\item $\bA ^d_{\bk}$: the affine space of dimension $d$. 
\item $\bP ^d_{\bk}$: the projective space of dimension $d$. 
\item $\bF _n$: the Hirzebruch surface of degree $n$; i.e., $\bF _n := \bP (\sO _{\bP ^1_{\bk}} \oplus \sO _{\bP ^1_{\bk}}(n))$. 
\item $\bA ^1_{\ast ,\bk}$: the affine line with one closed point removed, i.e., $\bA ^1_{\ast ,\bk} = \Spec (\bk [t^{\pm 1}])$. 
\item $\Cl (X)$: the divisor class group of $X$. 
\item $\Cl (X)_{\bQ} := \Cl (X) \otimes _{\bZ} \bQ$. 
\item $\rho (X)$: the Picard number of $X$. 
\item $K_X$: the canonical divisor on $X$. 
\item $\Sing (X)$: the set of singular points on $X$. 
\item $D_1\sim D_2$: $D_1$ and $D_2$ are linearly equivalent. 
\item $D_1\sim _{\bQ} D_2$: $D_1$ and $D_2$ are $\bQ$-linearly equivalent. 
\item $(D_1 \cdot D_2)$: the intersection number of $D_1$ and $D_2$. 
\item $(D)^2$: the self-intersection number of $D$. 
\item $\varphi ^{\ast}(D)$: the total transform of $D$ by a morphism $\varphi$. 
\item $\psi _{\ast}(D)$: the direct image of $D$ by a morphism $\psi$. 
\item $\Supp (D)$: the support of $D$. 
\item $|D|$: the complete linear system of $D$. 
\item $\sharp D$: the number of all irreducible components in $\Supp (D)$. 
\item $\delta _{i,j}$: The Kronecker delta. 
\end{itemize}
\end{notation}
\begin{acknowledgment}
The author would like to thank Doctor Jaehyun Kim and Doctor Dae-Won Lee for their valuable discussions and comments at Ewha Womans University. 
Moreover, he is also grateful to Professor Joonyeong Won for providing an excellent discussion place with them. 
Moreover, he would be to express the referees for careful reading this article and suggesting many useful comments. 
The author is supported by JSPS KAKENHI Grant Number JP24K22823 and JP25K17222. 
\end{acknowledgment}
\section{Preliminaries}\label{2}
\subsection{Basic results}
In this subsection, we summarize the following results, which are basic but important. 
\begin{lem}\label{lem(2-1-1)}
Let $\hM$ and $\hF$ be the minimal section and a general fiber on the Hirzebruch surface $\bF _n$ of degree $n$, respectively. 
Then the following assertions hold: 
\begin{enumerate}
\item Let $\hF_1,\dots ,\hF _r$ be distinct fibers on $\bF _n$ other than $\hF$. 
Then $\bF _n \backslash (\hM \cup \hF \cup \hF _1 \cup \dots \cup \hF _r) \simeq \bA ^1_{\bk} \times (\bA ^1_{\bk} \backslash \{ r\text{ points}\})$. 
\item Let $\hGamma$ be a smooth rational curve with $\hGamma \sim \hM + n\hF$. 
Then $\bF _n \backslash (\hGamma \cup \hM \cup \hF ) \simeq \bA ^1_{\bk} \times \bA ^1_{\ast ,\bk}$. 
\item Let $\hGamma$ be a smooth rational curve with $\hGamma \sim \hM + (n+1)\hF$. 
Then $\bF _n \backslash (\hGamma \cup \hM \cup \hF _0 ) \simeq \bA ^1_{\bk} \times \bA ^1_{\ast ,\bk}$, where $\hF _0$ is the fiber satisfying $\hM \cap \hGamma \cap \hF _0 \not= \emptyset$. 
\end{enumerate}
\end{lem}
\begin{proof}
See {\cite[Lemma 2.3]{Saw25}} (see also {\cite{Koj02}}). 
\end{proof}
\begin{lem}\label{lem(2-1-2)}
Let $\wS$ be a smooth projective surface with a $\bP ^1$-fibration $g:\wS \to \bP ^1_{\bk}$. 
Assume that $g$ admits a section $\wD _0$ and a singular fiber $\wF$, which consists only of $(-1)$-curves and $(-2)$-curves. 
Then the weighted dual graph of $\wF _i+ \wD _0$ is one of the following: 
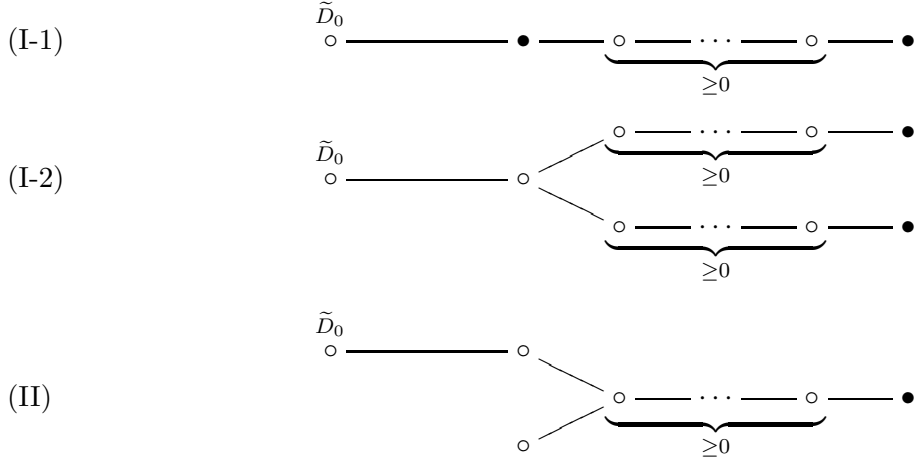
\begin{align}
\label{I-1} &\xygraph{\circ ([]!{+(0,.25)} {^{\wD_0}}) -[rr] \bullet -[r] \circ -[r] \cdots ([]!{+(0,-.35)} {\underbrace{\ \quad \qquad \qquad \qquad}_{\ge 0}}) -[r] \circ -[r] \bullet} \tag{I-1}\\ 
\label{I-2} &\xygraph{\circ ([]!{+(0,.25)} {^{\wD_0}}) -[rr] \circ (- []!{+(1,0.5)} \circ -[r] \cdots ([]!{+(0,-.35)} {\underbrace{\ \quad \qquad \qquad \qquad}_{\ge 0}}) -[r] \circ -[r] \bullet, - []!{+(1,-.5)} \circ -[r] \cdots ([]!{+(0,-.35)} {\underbrace{\ \quad \qquad \qquad \qquad}_{\ge 0}}) -[r] \circ -[r] \bullet)} \tag{I-2}\\ 
\label{II} &
\xygraph{\bullet -[l] \circ -[l] \cdots ([]!{+(0,-.4)} {\underbrace{\ \quad \qquad \qquad \qquad}_{\ge 0}}) -[l] \circ (-[]!{+(-1,-0.5)} \circ , -[]!{+(-1,0.5)} \circ -[ll] \circ ([]!{+(0,.25)} {^{\wD_0}}))} \tag{II}
\end{align}
Here, each vertex with label $\wD _0$ in the above graphs corresponds to the curve $\wD _0$. 
\end{lem}
\begin{proof}
See {\cite[Lemma 1.5]{Zha88}}. 
\end{proof}
\subsection{Du Val del Pezzo surfaces and Weak del Pezzo surfaces}\label{2-2}
A {\it del Pezzo surface} is a normal projective surface whose anti-canonical divisor is ample. 
Moreover, we say that a {\it Du Val del Pezzo surface} is a del Pezzo surface with at worst Du Val singularities. 
See, e.g., {\cite{Dur79}}, for details on Du Val singularities. 
For a Du Val del Pezzo surface $S$, ${\rm Dyn} (S)$ denotes the Dynkin type of its singularities; e.g., ${\rm Dyn}(S)  = \sA _2+2\sA _1$ implies that $\Sing (S)$ consists of a Du Val singular point of type $\sA _2$ and two Du Val singular points of type $\sA _1$. 
On the other hand, a {\it weak del Pezzo surface} is a smooth projective surface, whose anti-canonical divisor is nef and big. 
Notice that the minimal resolution of any Du Val del Pezzo surface is a weak del Pezzo surface.

We summarize some properties of weak del Pezzo surfaces. 
Let $\wS$ be a weak del Pezzo surface. 
We note $(-K_{\wS})^2 > 0$ and $(-K_{\wS} \cdot \wC) \ge 0$ for every curve $\wC$ on $\wS$ because $-K_{\wS}$ is nef and big. 
\begin{lem}\label{lem(2-2-1)}
$\wS \simeq \bP ^1_{\bk} \times \bP ^1_{\bk}$, $\wS \simeq \bF _2$, or there exists a birational morphism $h:\wS \to \bP ^2_{\bk}$. 
In particular, $\wS$ is rational. 
\end{lem}
\begin{proof}
See {\cite[Theorem 8.1.15]{Dol12}}. 
\end{proof}
By Lemma \ref{lem(2-2-1)} combined with $(-K_{\wS})^2 > 0$, we know that $(-K_{\wS})^2$ is an integer between $1$ and $9$. 
The number $(-K_{\wS})^2$ is called the {\it degree} of $\wS$. 
Moreover, we say that the {\it degree} of a Du Val del Pezzo surface is the degree of its minimal resolution. 

An irreducible curve $\wC$ on $\wS$ is a {\it negative curve} if $(\wC ) ^2 < 0$. 
Then the following lemma on negative curves holds: 
\begin{lem}\label{lem(2-2-2)}
The following assertions hold: 
\begin{enumerate}
\item Every negative curve of $\wS$ is either a $(-1)$-curve or a $(-2)$-curve. 
\item The number of all $(-2)$-curves on $\wS$ is at most $9-(-K_{\wS})^2$. 
\item For any irreducible curve $\wC$ on $\wS$, then $\wC$ is a $(-2)$-curve if and only if $(\wC \cdot -K_{\wS}) = 0$. 
\end{enumerate}
\end{lem}
\begin{proof}
In (1) and (2), see {\cite[Lemma 8.1.13]{Dol12}} and {\cite[Proposition 8.2.25]{Dol12}}, respectively. 

In (3), let $\wC$ be an irreducible curve on $\wS$. 
If $\wC$ is a $(-2)$-curve, then $(\wC \cdot -K_{\wS}) = 0$ by the adjunction formula. 
Conversely, assume that $(\wC \cdot -K_{\wS}) = 0$. 
Since $(-K_{\wS})^2 > 0$, we know that $\wC$ is a negative curve by the Hodge index theorem. 
Hence, $\wC$ is a $(-2)$-curve by virtue of (1) combined with $(\wC \cdot -K_{\wS}) = 0$. 
\end{proof}
Now, $d$ denotes the degree of $\wS$. 
By Lemma \ref{lem(2-2-2)} (2), there are at most finite $(-2)$-curves on $\wS$. 
Let $\wD$ be the reduced divisor consisting of all $(-2)$-curves on $\wS$. 
It is known that the dual graph of $\wD$ corresponds to a subsystem of the root systems of types $\sE _8$, $\sE _7$, $\sE _6$, $\sD _5$, $\sA _4$, $\sA _2 + \sA _1$ and $\sA _1$ for $d=1,\dots ,7$, respectively, with exceptions: $8\sA _1$, $7\sA _1$ and $\sD _4+4\sA _1$ for $d=1$ and $7\sA _1$ for $d=2$ (see {\cite{CPW16b,CT88,BW79,Ura81}}). 
In particular, the intersection matrix of $\wD$ is negative definite. 
Hence, we obtain the contraction $f:\wS \to S$ of $\wD$. 
Then we can easily see that $-K_S$ is ample. 
Namely, $S$ is a Du Val del Pezzo surface. 
Therefore, Du Val del Pezzo surfaces are in one-to-one correspondence with weak del Pezzo surfaces via minimal resolutions. 
\subsection{Singularity types of Du Val del Pezzo surfaces}\label{2-3}
Note that Du Val del Pezzo surfaces are classified (see, e.g., {\cite{CT88,BW79,Ura81}}). 
In this subsection, we recall a classification of Du Val del Pezzo surfaces. 

To consider types of Du Val del Pezzo surfaces, we prepare the following definition: 
\begin{defin}[{\cite[Definition 2.8]{CPW16b}}]
Let $S_1$ and $S_2$ be two Du Val del Pezzo surfaces, and let $f_1:\wS _1 \to S_1$ and $f_2:\wS _2 \to S_2$ be the minimal resolutions. 
Then we say that $S_1$ and $S_2$ have the {\it same singularity type} if there exists an isomorphism of $\Pic (\wS _1) \simeq \Pic (\wS _2)$ preserving the intersection form that gives a bijection between their sets of classes of negative curves. 
\end{defin}
Note that any negative curve on every weak del Pezzo surface is a $(-1)$-curve or a $(-2)$-curve by Lemma \ref{lem(2-2-2)} (1). 
Hence, for two Du Val del Pezzo surfaces $S_1$ and $S_2$ with the same type, we know $(-K_{S_1})^2 = (-K_{S_2})^2$ and ${\rm Dyn}(S_1) = {\rm Dyn}(S_2)$. 
However, the converse does not hold. 
In other words, there exist two Du Val del Pezzo surfaces $S_1$ and $S_2$ such that $(-K_{S_1})^2 = (-K_{S_2})^2$ and ${\rm Dyn}(S_1) = {\rm Dyn}(S_2)$ but they do not have the same type. 
More precisely, the pair of the degree and the Dynkin type of a Du Val del Pezzo surface is one of the following if and only if there exist two Du Val del Pezzo surfaces, whose pair of the degree and the Dynkin type is the same, not have the same singularity type: 
\begin{align*}
&(6,\sA _1),\\ 
&(4,\sA _3),\ (4,2\sA _1),\\ 
&(2,\sA _5+\sA _1),\ (2,\sA _5),\ (2,\sA _3+2\sA _1),\ (2,\sA _3+\sA _1),\ (2,4\sA _1),\ (2,3\sA _1),\\
&(1,\sA _7),\ (1,\sA _5+\sA _1),\ (1,2\sA _3),\ (1,\sA _3+2\sA _1),\ (1,4\sA _1). 
\end{align*}

Let $S$ be a Du Val del Pezzo surface of degree $2$ with ${\rm Dyn}(S) = \sA _5+\sA _1$, $\sA _5$, $\sA _3+2\sA _1$ or $\sA _3+\sA _1$, and let $f:\wS \to S$ be the minimal resolution. 
We shall introduce a notation to distinguish the singularity types of $S$. 

Assume that ${\rm Dyn}(S) = \sA _5+\sA _1$ or $\sA _5$. 
Let $\wD$ be the reduced divisor on $\wS$, which can be contracted to a singular point of type $\sA _5$, and let $\wD _0$ be the central component of $\wD$; i.e., $\wD _0$ is the irreducible component of $\wD$ such that $\wD - \wD _0$ can be contracted to two singular points of type $\sA _2$. 
Then $S$ is of type $(\sA _5+\sA _1)'$ (resp. $(\sA _5)'$) if ${\rm Dyn}(S) = \sA _5+\sA _1$ (resp. $\sA _5$) and there exists a $(-1)$-curve $\wE$ on $\wS$ such that $(\wD_0 \cdot \wE) = 1$. 
On the other hand, we say that $S$ is of type $(\sA _5+\sA _1)''$ (resp. $(\sA _5)''$) if ${\rm Dyn}(S) = \sA _5+\sA _1$ (resp. $\sA _5$) and there is no such $(-1)$-curve on $\wS$. 

In what follows, assume that ${\rm Dyn}(S) = \sA _3+2\sA _1$ or $\sA _3+\sA _1$. 
Let $\wD$ be the reduced divisor on $\wS$, which can be contracted to a singular point of type $\sA _3$, and let $\wD _0$ be the central component of $\wD$; i.e., $\wD _0$ is the irreducible component of $\wD$ such that $\wD - \wD _0$ can be contracted to two singular points of type $\sA _1$. 
Then $S$ is of type $(\sA _3+2\sA _1)'$ (resp. $(\sA _3+\sA _1)'$) if ${\rm Dyn}(S) = \sA _3+2\sA _1$ (resp. $\sA _3+\sA _1$) and there exists a $(-1)$-curve $\wE$ on $\wS$ such that $(\wD_0 \cdot \wE) = (\wE \cdot \wD_0') = 1$ for some a $(-2)$-curve $\wD_0'$ not contained in $\Supp (\wD)$. 
On the other hand, we say that $S$ is of type $(\sA _3+2\sA _1)''$ (resp. $(\sA _3+\sA _1)''$) if ${\rm Dyn}(S) = \sA _3+2\sA _1$ (resp. $\sA _3+\sA _1$) and there is no such $(-1)$-curve on $\wS$. 
\begin{rem}
The number of primes in the above notation reflects the number of special $(-1)$-curves on $\wS$. 
Indeed, if $\wS$ is of type $(\sA _5+\sA _1)'$ $(\sA _5)'$, $(\sA _3+2\sA _1)'$ or $(\sA _3+\sA _1)'$, then there exists a $(-1)$-curve on $\wS$ meeting the central component of $\wD$. 
On the other hand, if $\wS$ is of type $(\sA _5+\sA _1)''$ or $(\sA _5)''$, then there exist two $(-1)$-curves on $\wS$ such that they meet distinct two $(-2)$-curves on $\wD$, which meet the central component of $\wD$, respectively. 
Moreover, if $\wS$ is of type $(\sA _3+2\sA _1)''$ or $(\sA _3+\sA _1)''$, then there exist two $(-1)$-curves on $\wS$ meeting the central component of $\wD$. 
These results will be proved in the next section \S \ref{3}. 
\end{rem}
\section{$(-1)$-curves on Du Val del Pezzo surfaces of degree two}\label{3}

Let $S$ be a Du Val del Pezzo surface of degree $2$. 
Assume that $S$ has a singular point $P$ of type $\sA _n$ with $n \ge 3$. 
Let $f:\wS \to S$ be the minimal resolution, let $\wD$ be the reduced exceptional divisor of $f$, let $\wD = \sum _{i=1}^m \wD ^{(i)}$ be the decomposition of $\wD$ into connected components, and let $\wD ^{(i)} = \sum _{j=1}^{n(i)}\wD _j^{(i)}$ be the decomposition of $\wD ^{(i)}$ into irreducible components for $i=1,\dots ,m$. 
Without loss of generality, we can assume $f(\Supp (\wD ^{(1)})) = \{ P\}$. Namely, $n(1) = n$. 
For simplicity, we put $\wD _j := \wD _j^{(1)}$ for $j=1,\dots ,n$. 
\begin{prop}\label{prop(3-0)}
With the same notations as above, one of the following assertions holds: 
\begin{enumerate}
\item[(A)] There exist two $(-1)$-curves $\wE _1$ and $\wE _2$ on $\wS$ such that the weighted dual graph of $\wE _1 + \wE _2 + \wD ^{(1)}$ looks like that in Figure \ref{fig(3-1)} (A). 
\item[(B)] $n=5$ and there exists a $(-1)$-curve $\wE$ on $\wS$ such that the weighted dual graph of $\wE + \wD ^{(1)}$ looks like that in Figure \ref{fig(3-1)} (B). 
\item[(C)] $n=3$ and there exist an irreducible component $\wD _4$ of $\wD - \wD ^{(1)}$ and a $(-1)$-curve $\wE$ on $\wS$ such that $(\wD _4 \cdot \wD - \wD _4) = 0$ and the weighted dual graph of $\wE + \wD ^{(1)} + \wD _4$ looks like that in Figure \ref{fig(3-1)} (C). 
\end{enumerate}
\end{prop}
\begin{figure}
\begin{center}\scalebox{0.8}{\begin{tikzpicture}
\node at (0,2) {\large (A)};
\node at (3,1) {\xygraph{\circ -[r] \circ (-[d] \bullet, -[r] \cdots -[r] \circ (-[d] \bullet ,-[r] \circ )) }};

\node at (7,2) {\large (B)};
\node at (10,1) {\xygraph{\circ -[r] \circ -[r] \circ (-[d] \bullet ([]!{+(.3,0)} {\wE}) ,-[r] \circ -[r] \circ ) }};

\node at (14,2) {\large (C)};
\node at (15.75,0.4) {\xygraph{ \circ -[r] \circ (-[d] \bullet ([]!{+(.3,0)} {\wE}) -[d] \circ ([]!{+(.3,0)} {\wD _4}) ,-[r] \circ ) }};
\end{tikzpicture}}\end{center}

\caption{The weighted dual graphs in Proposition \ref{prop(3-0)}}\label{fig(3-1)}
\end{figure}
It seems that the result of Proposition \ref{prop(3-0)} is well-known to study Du Val del Pezzo surfaces of low degree but the author could not find any reference. 
Hence, the purpose of this section is to prove this proposition for the readers' convenience. 
Here, our argument is based on {\cite[\S 4]{Saw24}}. 

Note that every singular point on $S$ is cyclic by the classification of Du Val del Pezzo surfaces of degree $2$ (see, e.g., {\cite{Ura81} or {\cite[\S 8.7]{Dol12}}) because $S$ has the singular point of type $\sA _n$ with $n \ge 3$. 
Hence, for every $i=1,\dots ,m$ we may assume that the dual graph of $\wD ^{(i)}$ is the following: 
\begin{align*}
\xygraph{\circ ([]!{+(0,.25)} {^{\wD_1^{(i)}}}) -[r] \circ ([]!{+(0,.25)} {^{\wD_2^{(i)}}}) -[r] \cdots -[r] \circ ([]!{+(0,.25)} {^{\wD_n^{(i)}}}) }.
\end{align*}
The following two lemmas can be shown by the elementary argument. Hence, we omit the proofs of these lemmas. 
\begin{lem}[{\cite[Lemma 4.5]{Saw24}}]\label{lem(3-1)}
Let $a_1,\dots ,a_n$ be positive integers satisfying $a_j \ge 2$ for every $j=2,\dots ,n-1$, and set the effective $\bZ$-divisor $\wA := \sum _{j=1}^n a_j \wD _j$ on $\wS$. 
Then $(\wA ) ^2 \le -4$. 
Moreover, $(\wA ) ^2 = -4$ if and only if $a_1=a_n=1$ and $a_2 = \dots = a_{n-1} = 2$. 
\end{lem}
\begin{lem}[{\cite[Lemma 4.1]{Saw24}}]\label{lem(3-2)}
Fix an integer $i$ with $1 \le i \le m$ and an integer $\ell$ with $1 \le \ell \le n(i)$. 
If an effective $\bQ$-divisor $\wB _{\ell}^{(i)} := \sum _{j=1}^{n(i)}b_{i,j}\wD _j^{(i)}$ on $\wS$ satisfies $(-\wB _{\ell}^{(i)} \cdot \wD _j^{(i)}) = \delta _{j,\ell}$ for every $j=1,\dots ,n$, then we have:   
\begin{align*}
\wB _{\ell}^{(i)} = \frac{n-\ell +1}{n+1} \sum _{j=1}^{\ell}j\wD _j^{(i)} + \frac{\ell}{n+1} \sum _{j=1}^{n-\ell}j\wD _{n-j+1}^{(i)}
\end{align*} 
and: 
\begin{align*}
(\wB _{\ell}^{(i)})^2 = -\frac{(n-\ell +1)\ell}{n+1}. 
\end{align*}
\end{lem}
In Lemma \ref{lem(3-2)}, if $-(\wB _{\ell}^{(i)}\cdot \wD _j^{(i)}) = \delta _{j,\ell}$, then the value of $(\wB _{\ell}^{(i)})^2$ is explicitly summarized in Table \ref{A-list} depending on the values of $n(i)$ and $\ell$: 
\begin{table}
\begin{center}
\begin{tabular}{|c||c|c|c|c|c|c|c|} \hline

$n(i) \backslash \ell$ & $1$ & $2$ & $3$ & $4$ & $5$ & $6$ & $7$ \\ \hline \hline

$1$ & $-\frac{1}{2}$ & & & & & & \\ \hline
$2$ & $-\frac{2}{3}$ & $-\frac{2}{3}$ & & & & & \\ \hline
$3$ & $-\frac{3}{4}$ & $-1$ & $-\frac{3}{4}$ & & & & \\ \hline
$4$ & $-\frac{4}{5}$ & $-\frac{6}{5}$ & $-\frac{6}{5}$ & $-\frac{4}{5}$ & & & \\ \hline
$5$ & $-\frac{5}{6}$ & $-\frac{4}{3}$ & $-\frac{3}{2}$ & $-\frac{4}{3}$ & $-\frac{5}{6}$ & & \\ \hline
$6$ & $-\frac{6}{7}$ & $-\frac{10}{7}$ & $-\frac{12}{7}$ & $-\frac{12}{7}$ & $-\frac{10}{7}$ & $-\frac{6}{7}$ & \\ \hline
$7$ & $-\frac{7}{8}$ & $-\frac{3}{2}$ & $-\frac{15}{8}$ & $-2$ & $-\frac{15}{8}$ & $-\frac{3}{2}$ & $-\frac{7}{8}$ \\ \hline
\end{tabular}
\end{center}
\caption{The value of $(\wB _{\ell}^{(i)})^2$ in Lemma \ref{lem(3-2)}}\label{A-list} 
\end{table}

From now on, we set the divisor $\wDelta := -K_{\wS} - \wD _1 -2(\wD _2 + \dots + \wD _{n-1}) - \wD _n$ on $\wS$. 
Note that $\wS$ is rational by Lemma \ref{lem(2-2-1)}. 
By the Riemann-Roch theorem, we thus obtain the following lemma: 
\begin{lem}[{\cite[Lemma 4.6]{Saw24}}]\label{lem(3-3)}
With the same notations as above, we have $|\wDelta | \not= \emptyset$. 
\end{lem}
By Lemma \ref{lem(3-3)}, there exist two effective $\bZ$-divisors $\wDelta _+$ and $\wDelta _0$ on $\wS$ such that $\wDelta \sim \wDelta _+ + \wDelta _0$ and every irreducible component $\wC_+$ of $\wDelta _+$ (resp. every irreducible component $\wC_0$ of $\wDelta _0$) satisfies $(\wC _+ \cdot -K_{\wS})>0$ (resp. $(\wC _0 \cdot -K_{\wS})=0$). 
Note that all irreducible components of $\wDelta _0$ are $(-2)$-curves on $\wS$ by Lemma \ref{lem(2-2-2)} (3). 
Hence, we can write: 
\begin{align*}
\wDelta - \wDelta _0 = -K_{\wS} - \sum _{i=1}^m \wA ^{(i)}, 
\end{align*}
where $\wA ^{(i)}$ is an effective $\bZ$-divisor, which is a $\bZ$-linear combination of $\wD _1^{(i)},\dots ,\wD _{n(i)}^{(i)}$, on $\wS$. 
Note that: 
\begin{align}\label{(3-1)}
\wA ^{(1)} = (c_1+1)\wD _1 + (c_2+2)\wD _2 + \dots + (c_{n-1}+2)\wD _{n-1} + (c_n+1)\wD _n
\end{align}
for some non-negative integers $c_1,\dots ,c_n$. 
\begin{lem}[{\cite[Lemma 4.7]{Saw24}}]\label{lem(3-4)}
With the same notations as above, we have $(\wDelta _+)^2 \le -2$. 
\end{lem}
\begin{proof}
We can write: 
\begin{align*}
\wDelta _+ \sim \wDelta - \wDelta _0 = -K_{\wS} - \sum _{i=1}^m \wA ^{(i)}. 
\end{align*}
By (\ref{(3-1)}) and Lemma \ref{lem(3-1)}, we have $(\wA ^{(1)})^2 \le -4$. 
Hence, we have $(\wDelta _+ )^2 \le 2 - 4 + 0 = -2$ because the intersection matrix of $\wD - \wD ^{(1)}$ is negative definite. 
\end{proof}
Note that any irreducible component of $\wDelta _+$ has self-intersection number $\ge -1$ because it is not a $(-2)$-curve on $\wS$. 
Hence, $\wDelta _+$ is reducible by Lemma \ref{lem(3-4)}. 
Moreover, since $(\wDelta _+ \cdot -K_{\wS}) = (\wDelta \cdot -K_{\wS}) = 2$, there exist two irreducible curves $\wE _1$ and $\wE _2$ on $\wS$ such that $\wDelta _+ = \wE _1+\wE _2$. 
By Lemmas \ref{lem(2-2-2)} (1) and (3) and \ref{lem(3-4)}, $\wE _1$ and $\wE _2$ are $(-1)$-curves on $\wS$. 

Now, we shall deal with the result based on {\cite[Proposition 4.9]{Saw24}}. 
This result consists of Propositions \ref{prop(3-1)} and \ref{prop(3-2)}. 
\begin{prop}\label{prop(3-1)}
With the same notations as above, assume further that $\wE _1 \not= \wE _2$. 
Then we have the following assertions: 
\begin{enumerate}
\item $(\wE _1 \cdot \wE _2) = 0$. 
\item $\wDelta _0 = 0$; in other words, $\wDelta \sim \wE _1 + \wE _2$. 
\item $(\wE _1 + \wE _2 \cdot \wD _j) = \delta _{2,j} + \delta _{n-1,j}$ for $j = 1,\dots ,n$. 
\item $(\wE _1 \cdot \wD ^{(1)}) = (\wE _2 \cdot \wD ^{(1)}) = 1$. 
\end{enumerate}
\end{prop}
\begin{proof}
In (1), this assertion follows from Lemma \ref{lem(3-4)}. 

In (2), we can write: 
\begin{align*}
-K_{\wS} -\wE _1 - \wE _2 \sim \sum _{i=1}^m \wA ^{(i)}. 
\end{align*}
Note that $(-K_{\wS} -\wE _1 - \wE _2)^2 = -4$ by virtue of (1). 
Since the intersection matrix of $\wD - \wD ^{(1)}$ is negative definite, we have $(\wA ^{(1)})^2 \ge -4$. 
Hence, we know: 
\begin{align*}
\wA ^{(1)} = \wD _1 + 2\wD _2 + \dots + 2\wD _{n-1} + \wD _n
\end{align*}
by (\ref{(3-1)}) and Lemma \ref{lem(3-1)}; moreover, $\wA ^{(i)} = 0$ for every $i=2,\dots ,m$. Namely, $\wDelta _0 = 0$. 

In (3), we know $(\wE _1 + \wE _2 \cdot \wD _j) = (\wDelta \cdot \wD _j) = \delta _{2,j} + \delta _{n-1,j}$ for $j = 1,\dots ,n$ by virtue of (2). 

In (4), by using (1) and (2), we have: 
\begin{align*}
-1 = (\wE _1)^2 = (\wE _1 \cdot \wDelta - \wE _2) = (\wE _1 \cdot -K_{\wS}) - 2(\wE _1 \cdot \wD ^{(1)}) + (\wE _1 \cdot \wD _1 + \wD _n). 
\end{align*}
Hence, we obtain $(\wE _1 \cdot \wD ^{(1)}) = 1$ by virtue of (3). 
Similarly, we also obtain $(\wE _2 \cdot \wD ^{(1)}) = 1$. 
\end{proof}
By Proposition \ref{prop(3-1)} (1), (3) and (4), if $\wE _1 \not= \wE _2$, then we know that the weighted dual graph of $\wE _1 + \wE _2 + \wD ^{(1)}$ looks like that in Figure \ref{fig(3-1)} (A). 

From now on, we thus assume that $\wE _1 = \wE _2$. 
For simplicity, we put $\wE := \wE _1$. 
Since $2\wE \sim \wDelta - \wDelta _0$, we can write: 
\begin{align*}
\wE \sim _{\bQ} -\frac{1}{2}K_{\wS} - \sum _{i=1}^m \frac{1}{2}\wA ^{(i)}. 
\end{align*}
\begin{prop}\label{prop(3-2)}
With the same notations and assumptions as above, we have $n=5$ or $n=3$. 
Furthermore, the following assertions hold: 
\begin{itemize}
\item Assume that $n=5$. Then $(\wE \cdot \wD _j) = \delta _{3,j}$ for $j=1,\dots ,5$. 
Moreover, $2\wE \sim -K_{\wS} -\wD _1-2\wD_2 - 3\wD _3 - 2\wD _4 -\wD _5$. 
\item Assume that $n=3$. Then $(\wE \cdot \wD _j) = \delta _{2,j}$ for $j=1,2,3$. 
Moreover, there exists an irreducible component $\wD _4$ of $\wD$ such that $(\wD _4 \cdot \wD - \wD _4) = 0$ and $2\wE \sim -K_{\wS} -\wD _1-2\wD_2 - \wD _3 - \wD _4$. 
\end{itemize}
\end{prop}
\begin{proof}
For simplicity, we put $\wB ^{(i)} := \frac{1}{2}\wA ^{(i)}$ for $i=1,\dots ,m$. 
Since: 
\begin{align*}
\sum _{i=1}^m \wB ^{(i)} \sim _{\bQ} -\frac{1}{2}K_{\wS} - \wE, 
\end{align*}
we have: 
\begin{align*}
\sum _{i=1}^m (\wB ^{(i)})^2 = -\frac{3}{2}. 
\end{align*}
By Lemma \ref{lem(3-2)}, we know $(\wB ^{(i)})^2 \le -\frac{1}{2}$ for every $i = 1,\dots ,m$ (see also Table \ref{A-list}). 
Moreover, we notice $(\wB ^{(1)})^2 < -\frac{1}{2}$ because $n(1) = n \ge 3$. 
Hence, we may assume: 
\begin{align*}
\wE \sim _{\bQ} -\frac{1}{2}K_{\wS} - \wB ^{(1)} - \wB ^{(2)}. 
\end{align*}
We consider the following two cases separately. 
\smallskip

\noindent
{\bf Case 1:} ($\wB ^{(2)} = 0$). 
In this case, we have $(\wB ^{(1)})^2 = -\frac{3}{2}$. 
Hence, we know $n=5$ or $n=7$ by looking at Table \ref{A-list}. 
We consider the following two subcases separately. 
\smallskip

\noindent
{\bf Subcase 1-1:} ($n=5$). 
In this subcase, $(\wB ^{(1)} \cdot \wD _j) = \delta _{3,j}$ for $j=1,\dots ,5$ by looking at Table \ref{A-list}. 
Hence: 
\begin{align*}
\wB ^{(1)} \sim _{\bQ} \frac{1}{2}\wD _1 + \wD _2 + \frac{3}{2}\wD _3 + \wD _4 + \frac{1}{2}\wD _5
\end{align*}
by Lemma \ref{lem(3-2)}. 
In particular, $2\wE _1 \sim -K_{\wS} - \wD _1 -2 \wD _2 -3 \wD _3 -2 \wD _4 - \wD _5$. 
\smallskip

\noindent
{\bf Subcase 1-2:} ($n=7$). 
In this subcase, $(\wB ^{(1)} \cdot \wD _j) = \delta _{2,j}$ or $\delta _{5,j}$ for $j=1,\dots ,7$ by looking at Table \ref{A-list}. 
By the symmetry, we may assume $(\wB ^{(1)} \cdot \wD _j) = \delta _{2,j}$ for $j=1,\dots ,7$. 
Hence: 
\begin{align*}
\wB ^{(1)} \sim _{\bQ} \frac{3}{4}\wD _1 + \frac{3}{2} \wD _2 + \frac{5}{4}\wD _3 + \wD _4 + \frac{3}{4}\wD _5 + \frac{1}{2} \wD _6 + \frac{1}{4} \wD _7
\end{align*}
by Lemma \ref{lem(3-2)}. 
In particular, $2\wE _1 \sim \wDelta - \wDelta _0 = -K_{\wS} - \frac{3}{2} \wD _1 -3 \wD _2 -\frac{5}{2} \wD _3 -2 \wD _4 -\frac{3}{2} \wD _5 - \wD _6 -\frac{1}{2} \wD _7$. 
However, this is a contradiction that $\wDelta - \wDelta _0$ is a $\bZ$-divisor. 
Thus, this subcase does not take place. 
\smallskip

\noindent
{\bf Case 2:} ($\wB ^{(2)} \not= 0$). 
In this case, we have $(\wB ^{(1)})^2 + (\wB ^{(2)})^2 = -\frac{3}{2}$. 
Hence, the pair $((\wB ^{(1)})^2, (\wB ^{(2)})^2)$ equals one of $(-1,-\frac{1}{2})$, $(-\frac{5}{6},-\frac{2}{3})$ and $(-\frac{3}{4},-\frac{3}{4})$ by looking at Table \ref{A-list}, where note $n \ge 3$. 
We consider the following two subcases separately. 
\smallskip

\noindent
{\bf Subcase 2-1:} ($(\wB ^{(1)})^2 = -1$). 
In this subcase, $n=3$, $n(2) = 1$, $(\wB ^{(1)} \cdot \wD _j) = \delta _{2,j}$ for $j=1,2,3$, and $(\wB ^{(2)} \cdot \wD _1^{(2)}) = 1$ by looking at Table \ref{A-list}. 
We put $\wD _4 := \wD _1^{(2)}$. 
Then $(\wD _4 \cdot \wD - \wD _4) = 0$ because $n(2)=1$. 
Moreover: 
\begin{align*}
\wB ^{(1)} \sim _{\bQ} \frac{1}{2}\wD _1 + \wD _2 + \frac{1}{2}\wD _3, \quad \wB ^{(2)} = \frac{1}{2} \wD _4. 
\end{align*}
by Lemma \ref{lem(3-2)}. 
In particular, $2\wE _1 \sim -K_{\wS} - \wD _1 -2 \wD _2 - \wD _3 - \wD _4$. 
\smallskip

\noindent
{\bf Subcase 2-2:} ($(\wB ^{(1)})^2 \not= -1$). 
In this subcase, by the same argument as above, we see that $\wDelta - \wDelta _0$ is not a $\bZ$-divisor. 
However, this is a contradiction. 
Thus, this subcase does not take place. 
\smallskip

The proof of Proposition \ref{prop(3-2)} is thus completed. 
\end{proof}
By Proposition \ref{prop(3-2)}, if $\wE _1 = \wE _2$, then $n=5$ or $n=3$. 
Furthermore, if $n=5$ (resp. $n=3$), then we obtain the weighted dual graph of $\wE + \wD ^{(1)}$ (resp. $\wE + \wD ^{(1)} + \wD _4$) looks like that in Figure \ref{fig(3-1)} (B) (resp. (C)), where $\wE := \wE _1$. 

Therefore, the proof of Proposition \ref{prop(3-0)} is completed. 
\section{Construction of cylinders in rational surfaces}\label{4}
\subsection{Notation on some curves}\label{4-1}
Let $S$ be a normal projective rational surface such that $\Sing (S) \not= \emptyset$, let $f:\wS \to S$ be the minimal resolution, and let $\wD$ be the reduced exceptional divisor of $f$. 
Assume that there exists a $\bP ^1$-fibration $g:\wS \to \bP ^1_{\bk}$. 

In this section, we will construct polarized cylinders in $S$ under several conditions about the $\bP ^1$-fibration $g$ (see \S\S \ref{4-2}--\ref{4-7}). 
Hence, we shall prepare some notations in this subsection. 
\begin{defin}
Let the notation be the same as above. Then: 
\begin{enumerate}
\item We say that $g$ satisfies the condition $(\ast)$ if the following two conditions hold: 
\begin{itemize}
\item There exists exactly one irreducible component $\wD _0$, which is a section of $g$, of $\wD$. 
Moreover, $(-K_{\wS})^2 = 4-m_0$ holds, where $m_0 := -(\wD _0) \ge 2$. 
\item Any irreducible component of $\wD - \wD _0$ is contained in singular fibers of $g$ and is either a $(-1)$-curves or a $(-2)$-curves. 
\end{itemize}
\item We say that $g$ satisfies the condition $(\ast \ast)$ if the following two conditions hold: 
\begin{itemize}
\item There exist exactly two irreducible components $\wD _0$ and $\wD _{\infty}$, which are sections of $g$, of $\wD$. 
Moreover, $(-K_{\wS})^2 = 6-m_0-m_{\infty}$ holds, where $m_0 := -(\wD _0) \ge 2$ and $m_{\infty} := -(\wD _{\infty}) \ge 2$. 
\item Any irreducible component of $\wD - (\wD _0+\wD _{\infty})$ is contained in singular fibers of $g$ and is either a $(-1)$-curves or a $(-2)$-curves. 
\end{itemize}
\end{enumerate}
\end{defin}
In \S \ref{5}, we will show that almost Du Val del Pezzo surface of degree 2 admits a $\bP ^1$-fibration satisfying either condition $(\ast )$ or $(\ast \ast )$ (see Lemmas \ref{lem(4-1)}, \ref{lem(4-2)}, \ref{lem(4-3)}, \ref{lem(4-4)}, \ref{lem(4-5)} and \ref{lem(4-6)}). 

In what follows, assume that there exists a $\bP ^1$-fibration $g:\wS \to \bP ^1_{\bk}$ satisfying either condition $(\ast )$ or $(\ast \ast )$. 
Namely, $g$ admits a section $\wD _0$, which is an irreducible component of $\wD$, and $(\wD _0)^2 = -m_0$ holds. 
From now on, we prepare some notation to be used in \S\S \ref{4-2}--\ref{4-7}. 

First, $\wF$ denotes a general fiber of $g$. 
In what follows, we consider the notation for singular fibers of $g$. 
By Lemma \ref{lem(2-1-2)}, there are three kinds of singular fibers of $g$. 
Let $r$, $s$ and $t$ be the numbers of singular fibers of $g$ corresponding to (\ref{I-1}), (\ref{I-2}) and (\ref{II}), respectively, and let $\wF _1,\dots ,\wF _{r+s+t}$ be all singular fibers of $g$. 
Here, if $r>0$ (resp. $s>0$, $t>0$), singular fibers $\{ \wF _i\} _{1 \le i \le r}$ (resp. $\{ \wF _{r+i} \} _{1 \le i \le s}$, $\{ \wF _{r+s+i} \} _{1 \le i \le t}$) correspond to (\ref{I-1}) (resp. (\ref{I-2}), (\ref{II})). 

Suppose $r>0$. 
Then, for $i=1,\dots ,r$ let $\{ \wD _{i,\lambda}\} _{0 \le \lambda \le \alpha _i}$ be all irreducible components of $\wF _i$, where we assume that the weighted dual graph of $\wD _0 + \wF _i$ is as follows: 
\begin{align*}
\xygraph{\circ ([]!{+(0,-.3)} {^{-m_0}}) ([]!{+(0,.25)} {^{\wD_0}}) -[rr] \bullet ([]!{+(0,.25)} {^{\wD_{i,0}}}) - []!{+(1.5,0)} \circ ([]!{+(0,.25)} {^{\wD_{i,1}}}) - []!{+(1.5,0)} \cdots - []!{+(1.5,0)} \circ ([]!{+(0,.25)} {^{\wD_{i,\alpha _i-1}}}) - []!{+(1.5,0)} \bullet([]!{+(0,.25)} {^{\wD_{i,\alpha _i}}})}
\end{align*}
Hence, notice that $\sharp \wF _i = \alpha _i + 1$. Furthermore, we put $\wE _i' := \wD _{i,0}$ and $\wE _i := \wD _{i,\alpha _i}$. 

Suppose $s>0$. 
Then, for $i=1,\dots ,s$ let $\{ \wD _{r+i,\lambda}\} _{0 \le \lambda \le \beta _i + \beta _i'}$ be all irreducible components of $\wF _{r+i}$, where we assume that the weighted dual graph of $\wD _0 + \wF _{r+i}$ is as follows: 
\begin{align*}
\xygraph{\circ ([]!{+(0,-.3)} {^{-m_0}}) ([]!{+(0,.25)} {^{\wD_0}}) -[rr] \circ ([]!{+(0,.25)} {^{\wD_{r+i,0}}}) (- []!{+(1.5,0.5)} \circ ([]!{+(0,.25)} {^{\wD_{r+i,1}}}) - []!{+(1.5,0)} \cdots - []!{+(1.5,0)} \circ ([]!{+(0,.25)} {^{\wD_{r+i,\beta _i-1}}}) - []!{+(1.5,0)} \bullet ([]!{+(0,.25)} {^{\wD_{r+i,\beta _i}}}), - []!{+(1.5,-.5)} \circ ([]!{+(0,.25)} {^{\wD_{r+i,\beta _i+1}}}) - []!{+(1.5,0)} \cdots - []!{+(1.5,0)} \circ ([]!{+(0,.25)} {^{\wD_{r+i,\beta _i + \beta _i' -1}}}) - []!{+(1.5,0)} \bullet ([]!{+(0,.25)} {^{\wD_{r+i,\beta _i + \beta _i'}}}))}
\end{align*}
and assume further that $\beta _i \ge \beta _i'$. 
Hence, notice that $\sharp \wF _{r+i} = \beta _i + \beta _i' + 1$. Furthermore, we put $\wE _{r+i} := \wD _{r+i,\beta _i}$ and $\wE _{r+i}' := \wD _{r+i,\beta _i + \beta _i'}$. 

Suppose $t>0$. 
Then, for $i=1,\dots ,t$ let $\{ \wD _{r+s+i,\lambda}\} _{0 \le \lambda \le \gamma _i}$ be all irreducible components of $\wF _{r+s+i}$, where we assume that the weighted dual graph of $\wD _0 + \wF _{r+s+i}$ is as follows: 
\begin{align*}
\xygraph{\bullet ([]!{+(0,.25)} {^{\wD_{r+s+i,\gamma _i}}}) - []!{+(-1.5,0)} \circ ([]!{+(0,.25)} {^{\wD_{r+s+i,\gamma _i-1}}}) - []!{+(-1.5,0)} \cdots - []!{+(-1.5,0)} \circ ([]!{+(0,.25)} {^{\wD_{r+s+i,2}}}) (- []!{+(-1.5,-0.5)} \circ ([]!{+(0,.25)} {^{\wD_{r+s+i,1}}}),- []!{+(-1.5,0.5)} \circ ([]!{+(0,.25)} {^{\wD_{r+s+i,0}}}) -[ll] \circ ([]!{+(0,-.3)} {^{-m_0}}) ([]!{+(0,.25)} {^{\wD_0}}))}
\end{align*}
Hence, notice that $\sharp \wF _{r+s+i} = \gamma _i +1$. Furthermore, we put $\wE _{r+s+i} := \wD _{r+s+i,\gamma _i}$. 

Now, set $\alpha := \sum _{i=1}^r\alpha _i$, $\beta := \sum _{i=1}^s\beta _i$, $\beta ':= \sum _{i=1}^s\beta _i'$ and $\gamma := \sum _{i=1}^t\gamma _i$. 
Then we notice $(-K_{\wS})^2 = 8- (\alpha + \beta + \beta ' + \gamma ) = 6-m_0-m_{\infty}$, where $m_{\infty} := 2$ if $g$ satisfies $(\ast)$. 

Put $F := f_{\ast}(\wF)$, $E_i := f_{\ast}(\wE _i)$ for $i=1,\dots ,r+s+t$ and $E_j' := f_{\ast}(\wE _j')$ for $j=1,\dots ,r+s$ (if $r+s>0$). 
Then we know $E _i' \sim _{\bQ} F - E_i$ for $i=1,\dots ,r+s$ (if $r+s>0$) and $2E_{r+s+j} \sim _{\bQ} F$ for $j=1,\dots ,t$ (if $t>0$). 
\subsection{Type $\sD _n$ $(n \not = 5)$ or $\sE _n$ case}\label{4-2}
In this subsection, we keep the notation from \S \S \ref{4-1} and assume further that $g$ satisfies $(\ast)$ and one of the following conditions holds: 
\begin{enumerate}
\item $\gamma \ge 5$; 
\item $\gamma > 0$ and $\beta ' \ge 2$; 
\item $\gamma =0$ and $\beta ' \ge 3$, 
\end{enumerate}
Then we have the following result, which is the main result of this subsection: 
\begin{prop}\label{prop(4-1)}
With the same notations and assumptions as above, we have $\Ampc (S) = \Amp (S)$. 
\end{prop}
Proposition \ref{prop(4-1)} can be shown by a similar argument to {\cite[\S 4.1]{Saw25}}. 
For the reader's convenience, we present the proof of this proposition. 

Let $H \in \Amp (S)$. 
Since $\Amp (S)$ is contained in $\Cl (S)_{\bQ} = \bQ [F] \oplus \left( \bigoplus _{i=1}^r\bQ [E_i] \right) \oplus \left( \bigoplus _{j=1}^s\bQ [E_{r+j}] \right)$, we can write:
\begin{align*}
H \sim _{\bQ} aF + \sum _{i=1}^rb_iE_i + \sum _{j=1}^sc_jE_{r+j}
\end{align*}
for some rational numbers $a,b_1,\dots ,b_r,c_1,\dots ,c_s$. 
Then we note $b_i = -\alpha _i(H \cdot E_i)<0$ for $i=1,\dots ,r$ if $r>0$. 
Put $s' := \sharp \{ j \in \{1,\dots ,s\} \,|\, c_j <0 \}$, where $s' := 0$ if $s=0$. 
In what follows, we assume $c_j<0$ for $j=1,\dots ,s'$ when $s'>0$. 
\begin{lem}\label{lem(3-1-1)}
We set the divisor: 
\begin{align*}
\wDelta &:= 2\wD _0 + 2m_0\wF - \sum _{i=1}^r \sum _{\lambda =1}^{\alpha _i}2\lambda \wD _{i,\lambda} \\
&\qquad - \sum _{j=1}^{s'} \sum _{\mu =1}^{\beta _j}2\mu \wD _{r+j,\mu} 
- \sum _{j'=s'+1}^s \sum _{\mu' =1}^{\beta _{j'}'}2\mu' \wD _{r+j',\beta_{j'}+\mu'} 
- \sum _{k=1}^t\sum _{\nu =1}^{\gamma _k}\nu \wD _{r+s+k,\nu}
\end{align*}
on $\wS$. 
Then the following assertions hold: 
\begin{enumerate}
\item $\dim |\wDelta| \ge 3m_0+2 -3\alpha -3\beta -\gamma$. 
\item If $t=0$, then $\frac{1}{2}\wDelta$ is a $\bZ$-divisor and further $\dim |\frac{1}{2}\wDelta| \ge m_0+1-\alpha -\beta$. 
\end{enumerate}
\end{lem}
\begin{proof}
Note that: 
\begin{align*}
(\wDelta)^2 &= 4m_0 - \sum _{i=1}^r 4\alpha _i - \sum _{j=1}^{s'} 4\beta _j - \sum _{j'=s'+1}^s 4\beta _{j'}' - \sum _{k=1}^t \gamma _k \ge 4m_0 -4(\alpha + \beta) - \gamma ,\\
(\wDelta \cdot -K_{\wS}) &= 2(m_0+2) - \sum _{i=1}^r 2\alpha _i - \sum _{j=1}^{s'} 2\beta _j - \sum _{j'=s'+1}^s 2\beta _{j'}' - \sum _{k=1}^t \gamma _k \ge 2(m_0+2) -2(\alpha + \beta) - \gamma
\end{align*}
because $\beta _{j'} \ge \beta _{j'}'$ for every $j' = s'+1,\dots ,s$ (if $s' \not= s$). 
By the Riemann-Roch theorem and rationality of $\wS$, we have: 
\begin{align*}
\dim |\wDelta| \ge \frac{1}{2}(\wDelta \cdot \wDelta - K_{\wS}) \ge 3m_0+2 -3\alpha -3\beta -\gamma. 
\end{align*}
Hence, we obtain the assertion (1). 
In order to show the assertion (2), we assume $t=0$. 
Then $\frac{1}{2}\wDelta$ is obviously a $\bZ$-divisor on $\wS$ by the definition of $\wDelta$. 
The remaining assertion follows from the same argument as the above argument. 
\end{proof}
\begin{rem}
For the reader's convenience, we present figures of general members in linear systems related to Lemma \ref{lem(3-1-1)}. 
Let $\wDelta$ be the same as in Lemma \ref{lem(3-1-1)}. 
Then: 
\begin{itemize}
\item The configuration of a general member $\wC \in |\wDelta|$ is given as Figure \ref{C2} (a). 
\item Assuming $t=0$, the configuration of a general member $\wC \in |\frac{1}{2}\wDelta|$ is given as Figure \ref{C2} (b). 
\end{itemize}
\end{rem}
\begin{figure}
\begin{minipage}[c]{1\hsize}\begin{center}\scalebox{0.7}{\begin{tikzpicture}
\draw [thick] (-1,0)--(19,0);
\node at (-1.4,0) {$\wD _0$};
\node at (-2,6.3) {\Large (a)};
\node at (1.5,3.4) {\large $\cdots \cdots$};
\node at (5,3) {\Large $\cdots$};
\node at (10,2.7) {\Large $\cdots$};
\node at (15,3) {\large $\cdots \cdots$};
\draw [very thick] (-0.25,5.25)--(19,5.25);
\draw [very thick] (-0.25,5.75)--(13,5.75);
\draw [very thick] (-0.25,5.25) .. controls (-0.5,5.25) and (-0.5,5.75) .. (-0.25,5.75);
\node at (-0.8,5.5) {$\wC$};
\node at (0,-1.5) {};

\draw [dashed] (0,-0.25)--(0.5,1);
\draw (0.5,0.75)--(0,2);
\draw (0,1.75)--(0.5,3);
\draw (0.5,3.75)--(0,5);
\draw [dashed] (0,4.75)--(0.5,6);
\node at (0.5,3.45) {$\vdots$};
\node at (0,-0.6) {\footnotesize $\wE _1'$};
\node at (0.5,6.3) {\footnotesize $\wE _1$};

\draw [dashed] (2,-0.25)--(2.5,1);
\draw (2.5,0.75)--(2,2);
\draw (2,1.75)--(2.5,3);
\draw (2.5,3.75)--(2,5);
\draw [dashed] (2,4.75)--(2.5,6);
\node at (2.5,3.5) {$\vdots$};
\node at (2,-0.6) {\footnotesize $\wE _r'$};
\node at (2.5,6.3) {\footnotesize $\wE _r$};

\draw (3.5,-0.25)--(3.5,1);
\draw (3.5,1) .. controls (3.5,1.25) .. (3.75,1.25);
\draw (3.75,1.25)--(4.5,1.25);
\draw (3.75,1)--(4,1.75);
\draw (4.25,1)--(4.5,1.75);
\draw (3.75,2.25)--(4,1.5);
\draw (4.25,2.25)--(4.5,1.5);
\node at (3.75,2.75) {$\vdots$};
\node at (3.75,3.25) {$\vdots$};
\node at (4.25,2.75) {$\vdots$};
\node at (4.25,3.25) {$\vdots$};
\draw (3.75,3.5)--(4,4.25);
\draw (4.25,3.5)--(4.5,4.25);
\draw (3.75,4.75)--(4,4);
\draw [dashed] (4.25,4.75)--(4.5,4);
\draw [dashed] (3.75,4.25)--(4.25,6);
\node at (3.5,-0.6) {\footnotesize $\wD _{r+1,0}$};
\node at (4.9,4.5) {\footnotesize $\wE_{r+1}'$};
\node at (4.5,6.3) {\footnotesize $\wE _{r+1}$};

\draw (5.5,-0.25)--(5.5,1);
\draw (5.5,1) .. controls (5.5,1.25) .. (5.75,1.25);
\draw (5.75,1.25)--(6.5,1.25);
\draw (5.75,1)--(6,1.75);
\draw (6.25,1)--(6.5,1.75);
\draw (5.75,2.25)--(6,1.5);
\draw (6.25,2.25)--(6.5,1.5);
\node at (5.75,2.75) {$\vdots$};
\node at (5.75,3.25) {$\vdots$};
\node at (6.25,2.75) {$\vdots$};
\node at (6.25,3.25) {$\vdots$};
\draw (5.75,3.5)--(6,4.25);
\draw (6.25,3.5)--(6.5,4.25);
\draw (5.75,4.75)--(6,4);
\draw [dashed] (6.25,4.75)--(6.5,4);
\draw [dashed] (5.75,4.25)--(6.25,6);
\node at (5.5,-0.6) {\footnotesize $\wD _{r+s',0}$};
\node at (6.9,4.5) {\footnotesize $\wE_{r+s'}'$};
\node at (6.5,6.3) {\footnotesize $\wE _{r+s'}$};

\draw (8.25,-0.25)--(8.25,1);
\draw (8.25,1) .. controls (8.25,1.25) .. (8.5,1.25);
\draw (8.5,1.25)--(9.25,1.25);
\draw (8.5,1)--(8.75,1.75);
\draw (9,1)--(9.25,1.75);
\draw (8.5,2.25)--(8.75,1.5);
\draw (9,2.25)--(9.25,1.5);
\node at (8.5,2.75) {$\vdots$};
\node at (9,2.75) {$\vdots$};
\draw (8.5,3)--(8.75,3.75);
\draw (9,3)--(9.5,4.75);
\draw (8.5,4.25)--(8.75,3.5);
\draw [dashed] (8.75,4.75)--(8.5,4);
\draw [dashed] (9,6)--(9.5,4.25);
\node at (8.25,-0.6) {\footnotesize $\wD _{r+s'+1,0}$};
\node at (9.2,6.3) {\footnotesize $\wE_{r+s'+1}'$};
\node at (8,4.75) {\footnotesize $\wE _{r+s'+1}$};

\draw (10.75,-0.25)--(10.75,1);
\draw (10.75,1) .. controls (10.75,1.25) .. (11,1.25);
\draw (11,1.25)--(11.75,1.25);
\draw (11,1)--(11.25,1.75);
\draw (11.5,1)--(11.75,1.75);
\draw (11,2.25)--(11.25,1.5);
\draw (11.5,2.25)--(11.75,1.5);
\node at (11,2.75) {$\vdots$};
\node at (11.5,2.75) {$\vdots$};
\draw (11,3)--(11.25,3.75);
\draw (11.5,3)--(12,4.75);
\draw (11,4.25)--(11.25,3.5);
\draw [dashed] (11.25,4.75)--(11,4);
\draw [dashed] (11.5,6)--(12,4.25);
\node at (10.75,-0.6) {\footnotesize $\wD _{r+s,0}$};
\node at (11.5,6.3) {\footnotesize $\wE_{r+s}'$};
\node at (10.7,4.75) {\footnotesize $\wE _{r+s}$};

\draw (13,-0.25)--(13,1.5);
\draw (12.75,1.25)--(14,1.25);
\draw (13.75,1.5)--(13.75,0.25);
\draw (13.25,1)--(13.5,1.75);
\draw (13.25,2.25)--(13.5,1.5);
\node at (13.25,2.75) {$\vdots$};
\node at (13.25,3.25) {$\vdots$};
\draw (13.25,3.5)--(13.5,4.25);
\draw (13.25,4.75)--(13.5,4);
\draw [dashed] (13.25,4.25)--(13.75,6);
\node at (13,-0.6) {\footnotesize $\wD _{r+s+1,0}$};
\node at (14.7,1.25) {\footnotesize $\wD _{r+s+1,2}$};
\node at (14.6,0.5) {\footnotesize $\wD _{r+s+1,1}$};
\node at (13.75,6.3) {\footnotesize $\wE _{r+s+1}$};

\draw (16,-0.25)--(16,1.5);
\draw (15.75,1.25)--(17,1.25);
\draw (16.75,1.5)--(16.75,0.25);
\draw (16.25,1)--(16.5,1.75);
\draw (16.25,2.25)--(16.5,1.5);
\node at (16.25,2.75) {$\vdots$};
\node at (16.25,3.25) {$\vdots$};
\draw (16.25,3.5)--(16.5,4.25);
\draw (16.25,4.75)--(16.5,4);
\draw [dashed] (16.25,4.25)--(16.75,6);
\node at (16,-0.6) {\footnotesize $\wD _{r+s+t,0}$};
\node at (17.7,1.25) {\footnotesize $\wD _{r+s+t,2}$};
\node at (17.6,0.5) {\footnotesize $\wD _{r+s+t,1}$};
\node at (16.75,6.3) {\footnotesize $\wE _{r+s+t}$};
\end{tikzpicture}}\end{center}\end{minipage}

\begin{minipage}[c]{1\hsize}\begin{center}\scalebox{0.7}{\begin{tikzpicture}
\draw [thick] (-1,0)--(13,0);
\node at (-1.4,0) {$\wD _0$};
\node at (-2,6.3) {\Large (b)};
\node at (1.5,3.4) {\large $\cdots \cdots$};
\node at (5,3) {\Large $\cdots$};
\node at (10,2.7) {\Large $\cdots$};
\draw [very thick] (-0.25,5.55)--(13,5.5);
\node at (-0.6,5.5) {$\wC$};

\draw [dashed] (0,-0.25)--(0.5,1);
\draw (0.5,0.75)--(0,2);
\draw (0,1.75)--(0.5,3);
\draw (0.5,3.75)--(0,5);
\draw [dashed] (0,4.75)--(0.5,6);
\node at (0.5,3.45) {$\vdots$};
\node at (0,-0.6) {\footnotesize $\wE _1'$};
\node at (0.5,6.3) {\footnotesize $\wE _1$};

\draw [dashed] (2,-0.25)--(2.5,1);
\draw (2.5,0.75)--(2,2);
\draw (2,1.75)--(2.5,3);
\draw (2.5,3.75)--(2,5);
\draw [dashed] (2,4.75)--(2.5,6);
\node at (2.5,3.5) {$\vdots$};
\node at (2,-0.6) {\footnotesize $\wE _r'$};
\node at (2.5,6.3) {\footnotesize $\wE _r$};

\draw (3.5,-0.25)--(3.5,1);
\draw (3.5,1) .. controls (3.5,1.25) .. (3.75,1.25);
\draw (3.75,1.25)--(4.5,1.25);
\draw (3.75,1)--(4,1.75);
\draw (4.25,1)--(4.5,1.75);
\draw (3.75,2.25)--(4,1.5);
\draw (4.25,2.25)--(4.5,1.5);
\node at (3.75,2.75) {$\vdots$};
\node at (3.75,3.25) {$\vdots$};
\node at (4.25,2.75) {$\vdots$};
\node at (4.25,3.25) {$\vdots$};
\draw (3.75,3.5)--(4,4.25);
\draw (4.25,3.5)--(4.5,4.25);
\draw (3.75,4.75)--(4,4);
\draw [dashed] (4.25,4.75)--(4.5,4);
\draw [dashed] (3.75,4.25)--(4.25,6);
\node at (3.5,-0.6) {\footnotesize $\wD _{r+1,0}$};
\node at (4.9,4.5) {\footnotesize $\wE_{r+1}'$};
\node at (4.5,6.3) {\footnotesize $\wE _{r+1}$};

\draw (5.5,-0.25)--(5.5,1);
\draw (5.5,1) .. controls (5.5,1.25) .. (5.75,1.25);
\draw (5.75,1.25)--(6.5,1.25);
\draw (5.75,1)--(6,1.75);
\draw (6.25,1)--(6.5,1.75);
\draw (5.75,2.25)--(6,1.5);
\draw (6.25,2.25)--(6.5,1.5);
\node at (5.75,2.75) {$\vdots$};
\node at (5.75,3.25) {$\vdots$};
\node at (6.25,2.75) {$\vdots$};
\node at (6.25,3.25) {$\vdots$};
\draw (5.75,3.5)--(6,4.25);
\draw (6.25,3.5)--(6.5,4.25);
\draw (5.75,4.75)--(6,4);
\draw [dashed] (6.25,4.75)--(6.5,4);
\draw [dashed] (5.75,4.25)--(6.25,6);
\node at (5.5,-0.6) {\footnotesize $\wD _{r+s',0}$};
\node at (6.9,4.5) {\footnotesize $\wE_{r+s'}'$};
\node at (6.5,6.3) {\footnotesize $\wE _{r+s'}$};

\draw (8.25,-0.25)--(8.25,1);
\draw (8.25,1) .. controls (8.25,1.25) .. (8.5,1.25);
\draw (8.5,1.25)--(9.25,1.25);
\draw (8.5,1)--(8.75,1.75);
\draw (9,1)--(9.25,1.75);
\draw (8.5,2.25)--(8.75,1.5);
\draw (9,2.25)--(9.25,1.5);
\node at (8.5,2.75) {$\vdots$};
\node at (9,2.75) {$\vdots$};
\draw (8.5,3)--(8.75,3.75);
\draw (9,3)--(9.5,4.75);
\draw (8.5,4.25)--(8.75,3.5);
\draw [dashed] (8.75,4.75)--(8.5,4);
\draw [dashed] (9,6)--(9.5,4.25);
\node at (8.25,-0.6) {\footnotesize $\wD _{r+s'+1,0}$};
\node at (9.2,6.3) {\footnotesize $\wE_{r+s'+1}'$};
\node at (8,4.75) {\footnotesize $\wE _{r+s'+1}$};

\draw (10.75,-0.25)--(10.75,1);
\draw (10.75,1) .. controls (10.75,1.25) .. (11,1.25);
\draw (11,1.25)--(11.75,1.25);
\draw (11,1)--(11.25,1.75);
\draw (11.5,1)--(11.75,1.75);
\draw (11,2.25)--(11.25,1.5);
\draw (11.5,2.25)--(11.75,1.5);
\node at (11,2.75) {$\vdots$};
\node at (11.5,2.75) {$\vdots$};
\draw (11,3)--(11.25,3.75);
\draw (11.5,3)--(12,4.75);
\draw (11,4.25)--(11.25,3.5);
\draw [dashed] (11.25,4.75)--(11,4);
\draw [dashed] (11.5,6)--(12,4.25);
\node at (10.75,-0.6) {\footnotesize $\wD _{r+s,0}$};
\node at (11.5,6.3) {\footnotesize $\wE_{r+s}'$};
\node at (10.7,4.75) {\footnotesize $\wE _{r+s}$};
\end{tikzpicture}}
\caption{The configuration of $\wC$ in Lemma \ref{lem(3-1-1)}}\label{C2}
\end{center}\end{minipage}
\end{figure}
For simplicity, we set $d_{r,s'} := a + \sum _{i=1}^rb_i + \sum _{j=1}^{s'}c_j $, where we consider $\sum _{j=1}^{s'}c_j = 0$ if $s'=0$. 
\begin{lem}\label{lem(3-1-2)}
If one of the following conditions holds: 
\begin{enumerate}
\item $\gamma \ge 5$; 
\item $\gamma > 0$ and $\beta ' \ge 2$; 
\item $\gamma =0$ and $\beta ' \ge 3$, 
\end{enumerate}
then $d_{r,s'} >0$. 
\end{lem}
\begin{proof}
Let $\wDelta$ be the same as in Lemma \ref{lem(3-1-1)}. 
In this proof, we will use the formula $(-K_{\wS})^2 = 8 - (\alpha + \beta + \beta ' + \gamma ) = 4-m_0$. 

In (1) and (2), we then obtain $\dim |\wDelta| \ge 3\beta ' + 2\gamma -10 \ge 0$ by Lemma \ref{lem(3-1-1)} (1) combined with $-(\alpha + \beta) = -m_0 -4 + \beta ' + \gamma$. 
Here, we note that $\gamma \ge 2$ provided $\gamma > 0$. 
Hence, $|\wDelta| \not= \emptyset$, so that we can take a general member $\wC$ of $|\wDelta|$. 
Notice $f_{\ast}(\wC ) \sim _{\bQ} 2m_0F - \sum _{i=1}^r 2\alpha _i E_i - \sum _{j=1}^{s'} 2\beta _j E_{r+j} - \sum _{j'=s'+1}^s 2\beta _{j'}' E_{r+j'}' - \sum _{k=1}^t \gamma _k E_{r+s+k} \not\sim _{\bQ} 0$. 
Thus, we obtain $0 < \frac{1}{2}(H \cdot f_{\ast}(\wC )) = d_{r,s'}$ (see also Figure \ref{C2} (a)). 

In (3), we then obtain $\dim | \frac{1}{2}\wDelta| \ge -3+\beta ' \ge 0$ by Lemma \ref{lem(3-1-1)} (2) combined with $-(\alpha + \beta) = -m_0 -4 + \beta '$. 
Hence, $|\frac{1}{2}\wDelta| \not= \emptyset$, so that we can take a general member $\wC$ of $|\frac{1}{2}\wDelta|$. 
Notice $f_{\ast}(\wC ) \sim _{\bQ} m_0F - \sum _{i=1}^r \alpha _i E_i - \sum _{j=1}^{s'} \beta _j E_{r+j} - \sum _{j'=s'+1}^s \beta _{j'}' E_{r+j'}' \not\sim _{\bQ} 0$. 
Thus, we obtain $0 < (H \cdot f_{\ast}(\wC )) = d_{r,s'}$ (see also Figure \ref{C2} (b)). 
\end{proof}
Now, we shall consider $U := f(\wS \backslash \Supp (\wD _0 + \wF + \wF _1 +\dots + \wF _{r+s+t}))$. 
Notice that $U$ is a cylinder in $S$. 
Indeed: 
\begin{align*}
U \simeq \wS \backslash \Supp \left( \wD _0 + \wF + \wF _1 + \dots + \wF _{r+s+t} \right) \simeq \bA ^1_{\bk} \times (\bA ^1_{\bk} \backslash \{ (r+s+t)\text{ points}\})
\end{align*}
by Lemma \ref{lem(2-1-1)} (1). 
Then the following result holds: 
\begin{lem}\label{lem(3-1-3)}
If $d_{r,s'} >0$, then $U$ is an $H$-polar cylinder. 
\end{lem}
\begin{proof}
We take the effective $\bQ$-divisor:
\begin{align*}
D := &\sum _{i=1}^r (-b_i) E_i' + \sum _{j=1}^{s'} (-c_j) E_{r+j}' + \sum _{j'=s'+1}^s c_{j'} E_{r+j'} \\
&\quad +\frac{d_{r,s'}}{r+s+t+1} \left \{F + \sum _{k=1}^{r+s}(E_k+E_k') + \sum _{\ell =1}^t2E_{r+s+\ell} \right\}
\end{align*}
on $S$. 
Then we know $D \sim _{\bQ} H$ and $S \backslash \Supp (D) = U$. 
Thus, $U$ is an $H$-polar cylinder. 
\end{proof}
Proposition \ref{prop(4-1)} follows from Lemmas \ref{lem(3-1-2)} and \ref{lem(3-1-3)}. 
\subsection{Type $\sD _5$ case}\label{4-3}
In this subsection, we keep the notation from \S \S \ref{4-1} and assume further that $g$ satisfies $(\ast)$, $(s,t) = (0,1)$ and $\gamma = 4$. 
Namely, $m_0 = \alpha$. 
By a similar argument to {\cite[Lemma 3.4]{Saw25}}, there exists the $(-1)$-curve $\wGamma$ on $\wS$ such that $\wGamma \sim \wD _0 + m_0\wF -\sum _{i=1}^r \sum _{\lambda = 1}^{\alpha _i} \lambda \wD _{i,\lambda} -\sum _{\mu = 1}^4 \wD _{r+1,\mu}$.  
Note that the configuration of the $\bP ^1$-fibration $g:\wS \to \bP ^1_{\bk}$ is given as in Figure \ref{Case2}. 
\begin{figure}\begin{center}\scalebox{0.7}
{\begin{tikzpicture}
\draw [very thick] (-1,0)--(7,0);
\node at (-1.4,0) {$\wD _0$};
\node at (1.75,3) {\Large $\cdots \cdots$};
\draw [thick][dashed] (-1,5.5)--(7,5.5);
\node at (-1.4,5.5) {$\wGamma$};
\node at (7.6,0) {$-m_0$};

\draw [dashed] (0.5,-0.25)--(0,1);
\draw (0,0.75)--(0.5,2);
\draw (0.5,1.75)--(0,3);
\draw (0,3.75)--(0.5,5);
\draw [dashed] (0.5,4.75)--(0,6);
\node at (0,3.5) {$\vdots$};
\node at (0.5,-0.6) {\footnotesize $\wE _1'$};
\node at (0.7,1.25) {\footnotesize $\wD _{1,1}$};
\node at (-0.1,2.25) {\footnotesize $\wD _{1,2}$};
\node at (-0.5,4.3) {\footnotesize $\wD _{1,\alpha _1-1}$};
\node at (0,6.3) {\footnotesize $\wE _1$};

\draw [dashed] (4,-0.25)--(3.5,1);
\draw (3.5,0.75)--(4,2);
\draw (4,1.75)--(3.5,3);
\draw (3.5,3.75)--(4,5);
\draw [dashed] (4,4.75)--(3.5,6);
\node at (3.5,3.5) {$\vdots$};
\node at (4,-0.6) {\footnotesize $\wE _r'$};
\node at (4.2,1.25) {\footnotesize $\wD _{r,1}$};
\node at (3.4,2.25) {\footnotesize $\wD _{r,2}$};
\node at (3,4.3) {\footnotesize $\wD _{r,\alpha _r-1}$};
\node at (3.5,6.3) {\footnotesize $\wE _r$};

\draw (5.5,-0.25)--(6.5,2);
\draw (6.3,0.9)--(6.3,4.85);
\draw (6.5,3.75)--(5.5,6);
\node at (5.5,-0.6) {\footnotesize $\wD _{r+1,0}$};
\node at (5.5,6.3) {\footnotesize $\wD _{r+1,1}$};
\node at (6.9,2.5) {\footnotesize $\wD _{r+1,2}$};
\draw (5,3)--(7,3);
\draw [dashed] (5.5,1.5)--(5.5,4.5);
\node at (4.7,3.3) {\footnotesize $\wD _{r+1,3}$};
\node at (5,4.2) {\footnotesize $\wE _{r+1}$};
\end{tikzpicture}}
\caption{The configuration of $g : \wS \to \bP ^1_{\bk}$ in Subsection \ref{4-3}.}\label{Case2}
\end{center}\end{figure}
Then we have the following result, which is the main result of this subsection: 
\begin{prop}\label{prop(4-2)}
With the same notations and assumptions as above, we have $\Ampc (S) = \Amp (S)$. 
\end{prop}
In what follows, we shall prove the above result. 
Let $H \in \Amp (S)$. 
Since $\Amp (S)$ is contained in $\Cl (S)_{\bQ} = \bQ [F] \oplus \left( \bigoplus _{i=1}^r \bQ [E_i] \right)$, we can write: 
\begin{align*}
H \sim _{\bQ} aF + \sum _{i=1}^rb_iE_i
\end{align*}
for some rational numbers $a,b_1,\dots ,b_r$. 
Without loss of generality, we may assume $\frac{b_1}{\alpha _1} \le \frac{b_2}{\alpha _2} \le \dots \le \frac{b_r}{\alpha _r}$. 
Then the following lemma holds: 
\begin{lem}\label{lem(3-2-1)}
$2a + \sum _{i=1}^r2b_i - \frac{b_r}{\alpha_r} > 0$. 
\end{lem}
\begin{proof}
Consider the following divisor on $\wS$: 
\begin{align*}
\wDelta := 2\wD _0 + 2m_0\wF - \sum _{i=1}^r \sum _{\lambda = 1}^{\alpha _i} 2\lambda \wD _{i,\lambda} + \wE _r - \sum _{\mu = 1}^4\mu \wD _{r+1,\mu}
\end{align*}
Since $(\wDelta \cdot K_{\wS} - \wF) = -2(m_0-\alpha)-3 = -3$, $(\wDelta )^2 = 4(m_0-\alpha)-1 = -1$ and $(\wDelta \cdot -K_{\wS}) = 2(m_0-\alpha)+1 = 1$, we obtain $\dim |\wDelta| \ge 0$ by the Riemann-Roch theorem and rationality of $\wS$. 
Hence, by virtue of $(H \cdot E_i) = -\frac{b_i}{\alpha _i}$ for $i=1,\dots ,r$ we have: 
\begin{align*}
0 < (H \cdot f_{\ast}(\wDelta)) = a(F \cdot f_{\ast}(\wDelta)) + \sum _{i=1}^rb_i(E_i \cdot f_{\ast}(\wDelta)) = 2a + \sum _{i=1}^r2b_i - \frac{b_r}{\alpha_r}. 
\end{align*}
\end{proof}
We note $2E_{r+1} \sim _{\bQ} F$ because $\wD _{r+1,0} \sim \wF - \wD _{r+1,1} - 2 (\wD _{r+1,2} + \wD _{r+1,3} + \wE _{r+1})$. 
Put $\Gamma := f_{\ast}(\wGamma ) \sim _{\bQ} -\sum _{i=1}^r\alpha _iE_i + (2\alpha -1)E_{r+1}$. 
Note $-\frac{b_r}{\alpha _r} = (H \cdot E_r) > 0$. 

Assume $\frac{b_1}{\alpha _1} = \frac{b_r}{\alpha _r}$. 
For simplicity, we set $d_0 := 2a + \frac{2\alpha b_r}{\alpha_r} - \frac{b_r}{\alpha _r}$. 
Notice that: 
\begin{align*}
d_0 = 2a + \sum _{i=1}^r \frac{2\alpha _i b_r}{\alpha_r} - \frac{b_r}{\alpha _r} = 2a + \sum _{i=1}^r 2b_i - \frac{b_r}{\alpha _r} > 0
\end{align*}
by Lemma \ref{lem(3-2-1)}. 
Letting $\varepsilon$ be a positive rational number satisfying $\varepsilon < \frac{d_0}{2m_0-1}$, we take the effective $\bQ$-divisor: 
\begin{align*}
D := \left( -\frac{b_r}{\alpha _r} + \varepsilon \right) \Gamma + \sum _{i=1}^r\alpha _i  \varepsilon E_i + \left\{ d_0-(2m_0-1)\varepsilon \right\} E_{r+1}
\end{align*}
on $S$. 
Then we know $H \sim _{\bQ} D$ and: 
\begin{align*}
S \backslash \Supp (D) \simeq \wS \backslash \Supp \left( \wGamma + \wD _0 + \sum _{i=1}^r (\wF _i - \wE _i') + \wF _{r+1} \right) \simeq \bA ^1_{\bk} \times \bA ^1_{\ast , \bk}
\end{align*}
by Lemma \ref{lem(2-1-1)} (2). 
Thus, $H \in \Ampc (S)$. 

In what follows, we can assume $\frac{b_1}{\alpha _1} < \frac{b_r}{\alpha _r}$. 
Put $r' := \max \{ i \in \{ 1,\dots ,r-1\} \,|\, \frac{b_i}{\alpha _i} < \frac{b_r}{\alpha _r} \}$. 
For simplicity, we put $\alpha ' := \sum _{i=1}^{r'}\alpha _i$. 
Then $2m_0-1 - 2\alpha ' > 0$ by virtue of $m_0 = \alpha > \alpha '$. 
Moreover, we set $d_{r'} := 2a + \sum _{i=1}^{r'}2b_i + \frac{2(\alpha - \alpha ')b_r}{\alpha_r} - \frac{b_r}{\alpha _r}$. 
Notice that: 
\begin{align*}
d_{r'} = 2a + \sum _{i=1}^{r'}2b_i + \sum _{j=r'+1}^r \frac{2\alpha _j b_r}{\alpha_r} - \frac{b_r}{\alpha _r} = 2a + \sum _{i=1}^r 2b_i - \frac{b_r}{\alpha _r} > 0
\end{align*}
by Lemma \ref{lem(3-2-1)}. 
Letting $\varepsilon$ be a positive rational number satisfying $\varepsilon < \frac{d_{r'}}{2m_0-1-2\alpha '}$, we take the effective $\bQ$-divisor: 
\begin{align*}
D := \left( -\frac{b_r}{\alpha _r} + \varepsilon \right) \Gamma + \sum _{i=1}^{r'} \alpha _i \left( \frac{b_r}{\alpha _r} -\frac{b_i}{\alpha _i} - \varepsilon \right) E_i' + 
\sum _{j=r'+1}^r\alpha _j \varepsilon E_j  + \left\{ d_{r'}-(2m_0-1 - 2\alpha ') \varepsilon \right\} E_{r+1}
\end{align*}
on $S$. 
Then $H \sim _{\bQ} D$. 
Moreover, we know: 
\begin{align*}
S \backslash \Supp (D) \simeq \wS \backslash \Supp \left( \wGamma + \wD _0 + \sum _{i=1}^{r'} (\wF _i - \wE _i) + \sum _{j=r'+1}^r (\wF _j - \wE _j') + \wF _{r+1} \right) \simeq \bA ^1_{\bk} \times \bA ^1_{\ast , \bk}
\end{align*}
by Lemma \ref{lem(2-1-1)} (2). 
Thus, $H \in \Ampc (S)$. 

The proof of Proposition \ref{prop(4-2)} is thus completed. 
\subsection{Type $(\sA _5)'$ case}\label{4-4}
In this subsection, we keep the notation from \S \S \ref{4-1} and assume further that $g$ satisfies $(\ast \ast)$, $(s,t) = (0,1)$, $\gamma = 3$ and $(\wD_0 \cdot \wD_{\infty}) = 0$. 
Namely, $-m_{\infty} = m_0 - \alpha - 1$, so that $(\wD_{\infty} \cdot \wD_{r+1,1}) = 1$. 
Note that the configuration of the $\bP ^1$-fibration $g:\wS \to \bP ^1_{\bk}$ is given as in Figure \ref{Case4}. 
\begin{figure}\begin{center}\scalebox{0.7}
{\begin{tikzpicture}
\draw [very thick] (-1,0)--(7,0);
\node at (-1.4,0) {$\wD _0$};
\node at (1.75,3) {\Large $\cdots \cdots$};
\draw [very thick] (-1,5.5)--(7,5.5);
\node at (-1.4,5.5) {$\wD _{\infty}$};
\node at (7.6,0) {$-m_0$};
\node at (7.6,5.5) {$-m_{\infty}$};

\draw [dashed] (0.5,-0.25)--(0,1);
\draw (0,0.75)--(0.5,2);
\draw (0.5,1.75)--(0,3);
\draw (0,3.75)--(0.5,5);
\draw [dashed] (0.5,4.75)--(0,6);
\node at (0,3.5) {$\vdots$};
\node at (0.5,-0.6) {\footnotesize $\wE _1'$};
\node at (0.7,1.25) {\footnotesize $\wD _{1,1}$};
\node at (-0.1,2.25) {\footnotesize $\wD _{1,2}$};
\node at (-0.5,4.3) {\footnotesize $\wD _{1,\alpha _1-1}$};
\node at (0,6.3) {\footnotesize $\wE _1$};

\draw [dashed] (4,-0.25)--(3.5,1);
\draw (3.5,0.75)--(4,2);
\draw (4,1.75)--(3.5,3);
\draw (3.5,3.75)--(4,5);
\draw [dashed] (4,4.75)--(3.5,6);
\node at (3.5,3.5) {$\vdots$};
\node at (4,-0.6) {\footnotesize $\wE _r'$};
\node at (4.2,1.25) {\footnotesize $\wD _{r,1}$};
\node at (3.4,2.25) {\footnotesize $\wD _{r,2}$};
\node at (3,4.3) {\footnotesize $\wD _{r,\alpha _r-1}$};
\node at (3.5,6.3) {\footnotesize $\wE _r$};

\draw (5.5,-0.25)--(6.5,2);
\draw (6.3,0.9)--(6.3,4.85);
\draw (6.5,3.75)--(5.5,6);
\node at (5.5,-0.6) {\footnotesize $\wD _{r+1,0}$};
\node at (5.5,6.3) {\footnotesize $\wD _{r+1,1}$};
\node at (5.7,2.3) {\footnotesize $\wD _{r+1,2}$};
\draw [dashed] (5,3)--(7,3);
\node at (5,3.3) {\footnotesize $\wE _{r+1}$};
\end{tikzpicture}}
\caption{The configuration of $g : \wS \to \bP ^1_{\bk}$ in Subsection \ref{4-4}.}\label{Case4}
\end{center}\end{figure}
Then we have the following result, which is the main result of this subsection: 
\begin{prop}\label{prop(3-3)}
With the same notations and assumptions as above, we have $\Ampc (S) = \Amp (S)$. 
\end{prop}
In what follows, we shall prove the above result. 
Let $H \in \Amp (S)$. 
Since $\Amp (S)$ is contained in $\Cl (S)_{\bQ} = \bigoplus _{i=1}^r \bQ [E_i]$, we can write: 
\begin{align*}
H \sim _{\bQ} \sum _{i=1}^r a_iE_i
\end{align*}
for some rational numbers $a_1,\dots ,a_r$. 
Without loss of generality, we may assume $\frac{a_1}{\alpha _1} \le \frac{a_2}{\alpha _2} \le \dots \le \frac{a_r}{\alpha _r}$. 
Since $\alpha + 1 = m_0+m_{\infty}$ and $m_{\infty} \ge 2$, we have $\alpha _1 + \dots + \alpha _r > m_0$. 
We set $r' := \min \{ i \in \{ 1,\dots ,r\} \,|\, \alpha _1 + \dots + \alpha _i \ge m_0\}$. 
For simplicity, we put $\alpha ' := \sum _{i=1}^{r'-1}\alpha _i$, where $\alpha ' = 0$ provided $r'=1$. 
Then the following lemma holds: 
\begin{lem}\label{lem(3-4-1)}
$2(a_1+\dots +a_i) + \frac{\left(2(m_0-\alpha _1 - \dots - \alpha _i)-1\right)a_{r'}}{\alpha _{r'}} > 0$ for every $i=1,\dots ,r'-1$. 
Moreover, $\frac{a_{r'}}{\alpha _{r'}} > 0$. 
\end{lem}
\begin{proof}
Consider the following divisor on $\wS$: 
\begin{align*}
\wDelta &:= 2\wD _0 + 2m_0\wF \\
&\qquad- \sum _{i=1}^{r'} \sum _{\lambda = 1}^{\alpha _i} 2\lambda \wD _{i,\lambda} + \sum _{\mu = m_0-\alpha '}^{\alpha _{r'}}(2\mu - 2m_0 + 2\alpha ' + 1)\wD _{r',\mu} - (\wD _{r+1,1}+2\wD _{r+1,2}+3\wE _{r+1}). 
\end{align*}
Since $(\wDelta \cdot K_{\wS}-\wF) = -4 <0$, $(\wDelta)^2 = 0$ and $(\wDelta \cdot -K_{\wS}) = 2$, we obtain $\dim |\wDelta| \ge 1$ by the Riemann-Roch theorem and rationality of $\wS$. 

At first, we shall show the first assertion with $i=r'-1$. 
Since $(\wDelta \cdot \wD _0) = (\wDelta \cdot \wD _{\infty}) = (\wDelta \cdot \wD _{r+1,0}) = 0$, we have: 
\begin{align*}
(E _i \cdot f_{\ast}(\wDelta)) = \left\{ \begin{array}{cc}
2 & \text{if}\ i <r' \\ \frac{2(m_0-\alpha ')-1}{\alpha _{r'}} & \text{if}\ i =r' \\ 0 & \text{if}\ i >r'
\end{array} \right. 
\end{align*}
for $i=1,\dots ,r$. Hence: 
\begin{align*}
0 < (H \cdot f_{\ast}(\wDelta)) = \sum _{i=1}^ra_i(E_i \cdot f_{\ast}(\wDelta)) = 2(a_1 + \dots + a_{r'-1}) + \frac{\left(2(m_0-\alpha ')-1\right)a_{r'}}{\alpha _{r'}} . 
\end{align*}

In what follows, we shall show the first assertion for the general case. 
By using the above result combined with $\frac{a_{i+1}}{\alpha _{i+1}} \le \dots \le \frac{a_{r'-1}}{\alpha _{r'-1}} \le \frac{a_{r'}}{\alpha _{r'}}$, we have: 
\begin{align*}
0 &< 2(a_1 + \dots + a_{r'-1}) + \frac{\left(2(m_0-\alpha ')-1\right)a_{r'}}{\alpha _{r'}} \\
&\le 2(a_1 + \dots +a_i) + \frac{\left( 2(\alpha _{i+1} + \dots + \alpha _{r'-1} + m_0 - \alpha ')-1\right)a_{r'}}{\alpha _{r'}} \\
&= 2(a_1 + \dots +a_i) + \frac{\left( 2(m_0 - \alpha _1 -\dots - \alpha _i)-1\right)a_{r'}}{\alpha _{r'}}. 
\end{align*}

We shall prove the remaining assertion. 
Similarly, we have: 
\begin{align*}
0 < 2(a_1 + \dots + a_{r'-1}) + \frac{\left(2(m_0-\alpha ')-1\right)a_{r'}}{\alpha _{r'}} \le \frac{(2m_0-1)a_{r'}}{\alpha _{r'}}. 
\end{align*}
By virtue of $2m_0-1>0$, we thus obtain $\frac{a_{r'}}{\alpha _{r'}} > 0$. 
\end{proof}
Note $2E_{r+1} \sim _{\bQ} F$ and $\frac{2m_0-1}{2}F \sim _{\bQ} \sum _{i=1}^r\alpha _iE_i$ because $\wD _{r+1,0} \sim \wF -\wD _{r+1,1} - 2(\wE _{r+1} + \wD _{r+1,2})$ and $\wD _{\infty} \sim \wD _0 + m_0\wF - \sum _{i=1}^r \sum _{\lambda = 1}^{\alpha _i} \lambda \wD _{i,\lambda} -(\wD _{r+1,1}+\wD_{r+1,2}+\wE_{r+1})$. 

Assume $\frac{a_1}{\alpha _1} = \frac{a_{r'}}{\alpha _{r'}}$. 
Note $\frac{a_{r'}}{\alpha _{r'}} > 0$ by Lemma \ref{lem(3-4-1)}. 
Letting $\varepsilon$ be a positive rational number satisfying $\varepsilon < \frac{a_{r'}}{\alpha _{r'}}$, we take the effective $\bQ$-divisor: 
\begin{align*}
D := \sum _{i=1}^r\alpha _i \left( \frac{a_i}{\alpha _i} - \frac{a_{r'}}{\alpha _{r'}} + \varepsilon \right)E_i + (2m_0-1) \left( \frac{a_{r'}}{\alpha _{r'}} - \varepsilon \right)E_{r+1}
\end{align*}
on $S$. 
Then we know $H \sim _{\bQ} D$ and: 
\begin{align*}
S \backslash \Supp (D) \simeq \wS \backslash \Supp \left( \wD _0 + \wD _{\infty} + \sum _{i=1}^r (\wF _i - \wE _i') + \wF _{r+1} \right) \simeq \bA ^1_{\bk} \times \bA ^1_{\ast , \bk}
\end{align*}
by Lemma \ref{lem(2-1-1)} (2). 
Thus, $H \in \Ampc (S)$. 

In what follows, we can assume $\frac{a_1}{\alpha _1} < \frac{a_{r'}}{\alpha _{r'}}$. 
Put $r'' := \max \{ i \in \{ 1,\dots ,r'-1\} \,|\, \frac{a_i}{\alpha _i} < \frac{a_{r'}}{\alpha _{r'}} \}$. 
For simplicity, we put $\alpha '' := \sum _{i=1}^{r''}\alpha _i$ and $a'' := \sum _{i=1}^{r''}a_i$. 
Then we note $m_0 - \alpha '' > 0$. 
Moreover, by using Lemma \ref{lem(3-4-1)} we have: 
\begin{align*}
\frac{1}{2m_0-1-2\alpha''}\left( a'' + \frac{(2m_0-1-2\alpha'')a_{r'}}{\alpha _{r'}} \right) > 0. 
\end{align*}
Letting $\varepsilon$ be a positive rational number satisfying: 
\begin{align*}
\varepsilon < \min \left\{ \frac{a_{r'}}{\alpha _{r'}} - \frac{a_1}{\alpha _1},\ \frac{1}{2m_0-1-2\alpha''}\left( a'' + \frac{(2m_0-1-2\alpha'')a_{r'}}{\alpha _{r'}} \right) \right\} , 
\end{align*}
we take the effective $\bQ$-divisor: 
\begin{align*}
D &:= \sum _{i=1}^{r''} \alpha _i \left( \frac{a_{r'}}{\alpha _{r'}} -\frac{a_i}{\alpha _i} - \varepsilon \right)E_i' + 
\sum _{j=r''+1}^r\alpha _j \left( \frac{a_j}{\alpha _j} - \frac{a_{r'}}{\alpha _{r'}} + \varepsilon \right)E_j \\
&\qquad + \left\{ 2a'' + (2m_0-1-2\alpha'')\left( \frac{a_{r'}}{\alpha _{r'}}-\varepsilon \right) \right\} E_{r+1}
\end{align*}
on $S$. 
Then $H \sim _{\bQ} D$. 
Moreover, we know: 
\begin{align*}
S \backslash \Supp (D) \simeq \wS \backslash \Supp \left( \wD _0 + \wD _{\infty} + \sum _{i=1}^{r''} (\wF _i - \wE _i) + \sum _{j=r''+1}^r (\wF _j - \wE _j') + \wF _{r+1} \right) \simeq \bA ^1_{\bk} \times \bA ^1_{\ast , \bk}
\end{align*}
by Lemma \ref{lem(2-1-1)} (2). 
Thus, $H \in \Ampc (S)$. 

The proof of Proposition \ref{prop(3-3)} is thus completed. 
\subsection{Type $(\sA _3+\sA _1)'$ case}\label{4-5}
In this subsection, we keep the notation from \S \S \ref{4-1} and assume further that $g$ satisfies $(\ast \ast)$, $(s,t) = (0,1)$, $\gamma = 2$, $(\wD_0 \cdot \wD_{\infty}) = 0$ and $(\wD_{\infty} \cdot \wD_{r+1,0}) = 1$. 
Namely, $-m_{\infty} = m_0 - \alpha$. 
Note that the configuration of the $\bP ^1$-fibration $g:\wS \to \bP ^1_{\bk}$ is given as in Figure \ref{Case3}. 
\begin{figure}\begin{center}\scalebox{0.7}
{\begin{tikzpicture}
\draw [very thick] (-1,0)--(7,0);
\node at (-1.4,0) {$\wD _0$};
\node at (1.75,3) {\Large $\cdots \cdots$};
\draw [very thick] (-1,5.5)--(4.5,5.5);
\draw [very thick] (5,0.5)--(7,0.5);
\draw [very thick] (4.5,5.5) .. controls (4.75,5.5) and (4.75,0.5) .. (5,0.5);
\node at (-1.4,5.5) {$\wD _{\infty}$};
\node at (0,-1.5) {};
\node at (7.6,0) {$-m_0$};
\node at (7.6,0.5) {$-m_{\infty}$};

\draw [dashed] (0.5,-0.25)--(0,1);
\draw (0,0.75)--(0.5,2);
\draw (0.5,1.75)--(0,3);
\draw (0,3.75)--(0.5,5);
\draw [dashed] (0.5,4.75)--(0,6);
\node at (0,3.5) {$\vdots$};
\node at (0.5,-0.6) {\footnotesize $\wE _1'$};
\node at (0.7,1.25) {\footnotesize $\wD _{1,1}$};
\node at (-0.1,2.25) {\footnotesize $\wD _{1,2}$};
\node at (-0.5,4.3) {\footnotesize $\wD _{1,\alpha _1-1}$};
\node at (0,6.3) {\footnotesize $\wE _1$};

\draw [dashed] (4,-0.25)--(3.5,1);
\draw (3.5,0.75)--(4,2);
\draw (4,1.75)--(3.5,3);
\draw (3.5,3.75)--(4,5);
\draw [dashed] (4,4.75)--(3.5,6);
\node at (3.5,3.5) {$\vdots$};
\node at (4,-0.6) {\footnotesize $\wE _r'$};
\node at (4.2,1.25) {\footnotesize $\wD _{r,1}$};
\node at (3.4,2.25) {\footnotesize $\wD _{r,2}$};
\node at (3,4.3) {\footnotesize $\wD _{r,\alpha _r-1}$};
\node at (3.5,6.3) {\footnotesize $\wE _r$};

\draw (5.5,-0.25)--(6.5,2);
\draw [dashed] (6.3,0.9)--(6.3,4.85);
\draw (6.5,3.75)--(5.5,6);
\node at (5.5,-0.6) {\footnotesize $\wD _{r+1,0}$};
\node at (5.5,6.3) {\footnotesize $\wD _{r+1,1}$};
\node at (5.8,2.5) {\footnotesize $\wE _{r+1}$};
\end{tikzpicture}}
\caption{The configuration of $g : \wS \to \bP ^1_{\bk}$ in Subsection \ref{4-5}.}\label{Case3}
\end{center}\end{figure}
Then we have the following result, which is the main result of this subsection: 
\begin{prop}\label{prop(3-4)}
With the same notations and assumptions as above, we have $\Ampc (S) = \Amp (S)$. 
\end{prop}
In what follows, we shall prove the above result. 
Let $H \in \Amp (S)$. 
Since $\Amp (S)$ is contained in $\Cl (S)_{\bQ} = \bigoplus _{i=1}^r \bQ [E_i]$, we can write: 
\begin{align*}
H \sim _{\bQ} \sum _{i=1}^r a_iE_i
\end{align*}
for some rational numbers $a_1,\dots ,a_r$. 
Without loss of generality, we may assume $\frac{a_1}{\alpha _1} \le \frac{a_2}{\alpha _2} \le \dots \le \frac{a_r}{\alpha _r}$. 
Since $\alpha = m_0+m_{\infty}$ and $m_{\infty} >0$, we have $\alpha _1 + \dots + \alpha _r \ge m_0$. 
We set $r' := \min \{ i \in \{ 1,\dots ,r\} \,|\, \alpha _1 + \dots + \alpha _i \ge m_0\}$. 
For simplicity, we put $\alpha ' := \sum _{i=1}^{r'-1}\alpha _i$, where $\alpha ' = 0$ provided $r'=1$. 
Then the following lemma holds: 
\begin{lem}\label{lem(3-3-1)}
$a_1+\dots +a_i + \frac{(m_0-\alpha _1 - \dots - \alpha _i)a_{r'}}{\alpha _{r'}} > 0$ for every $i=1,\dots ,r'-1$. 
Moreover, $\frac{a_{r'}}{\alpha _{r'}} > 0$. 
\end{lem}
\begin{proof}
Consider the following divisor on $\wS$: 
\begin{align*}
\wDelta := \wD _0 + m_0\wF - \sum _{i=1}^{r'} \sum _{\lambda = 1}^{\alpha _i} \lambda \wD _{i,\lambda} + \sum _{\mu = m_0-\alpha '}^{\alpha _{r'}}(\mu - m_0 + \alpha ')\wD _{r',\mu} - (\wD _{r+1,1}+\wE _{r+1}). 
\end{align*}
Since $(\wDelta \cdot K_{\wS}-\wF) = -2 <0$, $(\wDelta)^2 = -1$ and $(\wDelta \cdot -K_{\wS}) = 1$, we obtain $\dim |\wDelta| \ge 0$ by the Riemann-Roch theorem and rationality of $\wS$. 

At first, we shall show the first assertion with $i=r'-1$. 
Since $(\wDelta \cdot \wD _0) = (\wDelta \cdot \wD _{\infty}) = (\wDelta \cdot \wD _{r+1,0}) = 0$, we have: 
\begin{align*}
(E _i \cdot f_{\ast}(\wDelta)) = \left\{ \begin{array}{cc}
1 & \text{if}\ i <r' \\ \frac{m_0-\alpha '}{\alpha _{r'}} & \text{if}\ i =r' \\ 0 & \text{if}\ i >r'
\end{array} \right. 
\end{align*}
for $i=1,\dots ,r$. Hence: 
\begin{align*}
0 < (H \cdot f_{\ast}(\wDelta)) = \sum _{i=1}^ra_i(E_i \cdot f_{\ast}(\wDelta)) = a_1 + \dots + a_{r'-1} + \frac{(m_0-\alpha ')a_{r'}}{\alpha _{r'}}. 
\end{align*}

In what follows, we shall show the first assertion for the general case. 
By using the above result combined with $\frac{a_{i+1}}{\alpha _{i+1}} \le \dots \le \frac{a_{r'-1}}{\alpha _{r'-1}} \le \frac{a_{r'}}{\alpha _{r'}}$, we have: 
\begin{align*}
0 &< a_1+\dots +a_{r'-1} + \frac{(m_0-\alpha')a_{r'}}{\alpha _{r'}} \\
&\le a_1 + \dots +a_i + \frac{(\alpha _{i+1} + \dots + \alpha _{r'-1} + m_0 - \alpha ')a_{r'}}{\alpha _{r'}} \\
&= a_1 + \dots +a_i + \frac{(m_0- \alpha _1 - \dots - \alpha _i)a_{r'}}{\alpha _{r'}}. 
\end{align*}

We shall prove the remaining assertion.
Similarly, we have: 
\begin{align*}
0 < a_1+\dots +a_{r'-1} + \frac{(m_0-\alpha')a_{r'}}{\alpha _{r'}} \le \frac{m_0a_{r'}}{\alpha _{r'}}. 
\end{align*}
By virtue of $m_0 \ge 2$, we thus obtain $\frac{a_{r'}}{\alpha _{r'}} > 0$. 
\end{proof}
Note $2E_{r+1} \sim _{\bQ} F$ and $m_0F \sim _{\bQ} \sum _{i=1}^r\alpha _iE_i$ because $\wD _{r+1,0} \sim \wF -\wD _{r+1,1} - 2\wE _{r+1}$ and $\wD _{\infty} \sim \wD _0 + m_0\wF - \sum _{i=1}^r \sum _{\lambda = 1}^{\alpha _i} \lambda \wD _{i,\lambda}$. 

Assume $\frac{a_1}{\alpha _1} = \frac{a_{r'}}{\alpha _{r'}}$. 
Then we note $\frac{a_{r'}}{\alpha _{r'}} > 0$ by Lemma \ref{lem(3-3-1)}. 
Letting $\varepsilon$ be a positive rational number satisfying $\varepsilon < \frac{a_{r'}}{\alpha _{r'}}$, we take the effective $\bQ$-divisor: 
\begin{align*}
D := \sum _{i=1}^r\alpha _i \left( \frac{a_i}{\alpha _i} - \frac{a_{r'}}{\alpha _{r'}} + \varepsilon \right)E_i + 2m_0 \left( \frac{a_{r'}}{\alpha _{r'}} - \varepsilon \right)E_{r+1}
\end{align*}
on $S$. 
Then we know $H \sim _{\bQ} D$ and: 
\begin{align*}
S \backslash \Supp (D) \simeq \wS \backslash \Supp \left( \wD _0 + \wD _{\infty} + \sum _{i=1}^r (\wF _i - \wE _i') + \wF _{r+1} \right) \simeq \bA ^1_{\bk} \times \bA ^1_{\ast , \bk}
\end{align*}
by Lemma \ref{lem(2-1-1)} (2). 
Thus, $H \in \Ampc (S)$. 

In what follows, we can assume $\frac{a_1}{\alpha _1} < \frac{a_{r'}}{\alpha _{r'}}$. 
Put $r'' := \max \{ i \in \{ 1,\dots ,r'-1\} \,|\, \frac{a_i}{\alpha _i} < \frac{a_{r'}}{\alpha _{r'}} \}$. 
For simplicity, we put $\alpha '' := \sum _{i=1}^{r''}\alpha _i$ and $a'' := \sum _{i=1}^{r''}a_i$. 
Then we note $m_0 - \alpha '' > 0$. 
Moreover, by Lemma \ref{lem(3-3-1)} we have: 
\begin{align*}
\frac{1}{m_0-\alpha''}\left( a'' + \frac{(m_0-\alpha'')a_{r'}}{\alpha _{r'}} \right) > 0. 
\end{align*}
Letting $\varepsilon$ be a positive rational number satisfying: 
\begin{align*}
\varepsilon < \min \left\{ \frac{a_{r'}}{\alpha _{r'}} - \frac{a_1}{\alpha _1},\ \frac{1}{m_0-\alpha''}\left( a'' + \frac{(m_0-\alpha'')a_{r'}}{\alpha _{r'}} \right) \right\} , 
\end{align*}
we take the effective $\bQ$-divisor: 
\begin{align*}
D &:= \sum _{i=1}^{r''} \alpha _i \left( \frac{a_{r'}}{\alpha _{r'}} -\frac{a_i}{\alpha _i} - \varepsilon \right)E_i' + 
\sum _{j=r''+1}^r\alpha _j \left( \frac{a_j}{\alpha _j} - \frac{a_{r'}}{\alpha _{r'}} + \varepsilon \right)E_j \\
&\qquad + 2\left\{ a'' + (m_0-\alpha '')\left( \frac{a_{r'}}{\alpha _{r'}}-\varepsilon \right) \right\}E_{r+1}
\end{align*}
on $S$. 
Then $H \sim _{\bQ} D$. 
Moreover, we know: 
\begin{align*}
S \backslash \Supp (D) \simeq \wS \backslash \Supp \left( \wD _0 + \wD _{\infty} + \sum _{i=1}^{r''} (\wF _i - \wE _i) + \sum _{j=r''+1}^r (\wF _j - \wE _j') + \wF _{r+1} \right) \simeq \bA ^1_{\bk} \times \bA ^1_{\ast , \bk}
\end{align*}
by Lemma \ref{lem(2-1-1)} (2). 
Thus, $H \in \Ampc (S)$. 

The proof of Proposition \ref{prop(3-4)} is thus completed. 
\subsection{Type $\sA _n$ $(n \ge 3)$ case}\label{4-6}
In this subsection, we keep the notation from \S \S \ref{4-1} and assume further that $g$ satisfies $(\ast \ast)$, $(s,t) = (1,0)$, $\beta' = 1$, $(\wD_0 \cdot \wD_{\infty}) = 0$ and $(\wD_{\infty} \cdot \wD_{r+1,\beta-1}) = 1$. 
Namely, $-m_{\infty} = m_0 - \alpha - (\beta -1)$. 
Note that the configuration of the $\bP ^1$-fibration $g:\wS \to \bP ^1_{\bk}$ is given as in Figure \ref{Case6}. 
\begin{figure}\begin{center}\scalebox{0.7}
{\begin{tikzpicture}
\draw [very thick] (-1,0)--(7,0);
\node at (-1.4,0) {$\wD _0$};
\node at (1.75,3) {\Large $\cdots \cdots$};
\draw [very thick] (-1,5.5)--(4.5,5.5);
\draw [very thick] (5,4.5)--(7,4.5);
\draw [very thick] (4.5,5.5) .. controls (4.75,5.5) and (4.75,4.5) .. (5,4.5);
\node at (-1.4,5.5) {$\wD _{\infty}$};
\node at (0,-1.5) {};
\node at (7.6,0) {$-m_0$};
\node at (7.6,4.5) {$-m_{\infty}$};

\draw [dashed] (0.5,-0.25)--(0,1);
\draw (0,0.75)--(0.5,2);
\draw (0.5,1.75)--(0,3);
\draw (0,3.75)--(0.5,5);
\draw [dashed] (0.5,4.75)--(0,6);
\node at (0,3.5) {$\vdots$};
\node at (0.5,-0.6) {\footnotesize $\wE _1'$};
\node at (0.7,1.25) {\footnotesize $\wD _{1,1}$};
\node at (-0.1,2.25) {\footnotesize $\wD _{1,2}$};
\node at (-0.5,4.3) {\footnotesize $\wD _{1,\alpha _1-1}$};
\node at (0,6.3) {\footnotesize $\wE _1$};

\draw [dashed] (4,-0.25)--(3.5,1);
\draw (3.5,0.75)--(4,2);
\draw (4,1.75)--(3.5,3);
\draw (3.5,3.75)--(4,5);
\draw [dashed] (4,4.75)--(3.5,6);
\node at (3.5,3.5) {$\vdots$};
\node at (4,-0.6) {\footnotesize $\wE _r'$};
\node at (4.2,1.25) {\footnotesize $\wD _{r,1}$};
\node at (3.4,2.25) {\footnotesize $\wD _{r,2}$};
\node at (3,4.3) {\footnotesize $\wD _{r,\alpha _r-1}$};
\node at (3.5,6.3) {\footnotesize $\wE _r$};

\draw (5.5,-0.25)--(6.5,1.5);
\draw [dashed] (6.5,0.25)--(5.5,1);
\draw (6.5,1.25)--(5.5,3);
\draw (5.5,3.75)--(6.5,5.5);
\draw [dashed] (6.5,4.75)--(5.5,6);
\node at (5.5,3.4) {$\vdots$};
\node at (5.5,-0.6) {\footnotesize $\wD_ {r+1,0}$};
\node at (6.7,0.6) {\footnotesize $\wE _{r+1}'$};
\node at (6.5,2.5) {\footnotesize $\wD _{r+1,1}$};
\node at (6.5,4) {\footnotesize $\wD _{r+1,\beta -1}$};
\node at (5.5,6.3) {\footnotesize $\wE _{r+1}$};
\end{tikzpicture}}
\caption{The configuration of $g : \wS \to \bP ^1_{\bk}$ in Subsection \ref{4-6}.}\label{Case6}
\end{center}\end{figure}
Then we have the following result, which is the main result of this subsection: 
\begin{prop}\label{prop(3-6)}
With the same notations and assumptions as above, we have the following: 
\begin{enumerate}
\item If $\beta = 1$, then $\Ampc (S) = \Amp (S)$. 
\item If $\beta \ge 2$ and $m_0=2$, then $\Ampc (S) = \Amp (S)$. 
\end{enumerate}
\end{prop}
In what follows, we shall prove the above result. 
Let $H \in \Amp (S)$. 
Since $\Amp (S)$ is contained in $\Cl (S)_{\bQ} = \bigoplus _{i=1}^{r+1} \bQ [E_i]$, we can write: 
\begin{align*}
H \sim _{\bQ} \sum _{i=1}^r a_iE_i + bE_{r+1}
\end{align*}
for some rational numbers $a_1,\dots ,a_r,b$. 
Without loss of generality, we may assume $\frac{a_1}{\alpha _1} \le \frac{a_2}{\alpha _2} \le \dots \le \frac{a_r}{\alpha _r}$. 

At first, we consider the case $\beta = 1$. 
Since $\alpha = m_0+m_{\infty}$ and $m_{\infty} > 0$, we have $\alpha _1 + \dots + \alpha _r \ge m_0$. 
We set $r' := \min \{ i \in \{ 1,\dots ,r\} \,|\, \alpha _1 + \dots + \alpha _i \ge m_0\}$. 
For simplicity, we put $\alpha ' := \sum _{i=1}^{r'-1}\alpha _i$, where $\alpha ' = 0$ provided $r'=1$. 
Then we obtain the following lemma: 
\begin{lem}\label{lem(3-6-1)}
Assume that $\beta  = 1$. Then the following assertions hold: 
\begin{enumerate}
\item $a_1+\dots +a_i + \frac{(m_0-\alpha _1 - \dots - \alpha _i)a_{r'}}{\alpha _{r'}} > 0$ for every $i=1,\dots ,r'-1$. 
Moreover, $\frac{a_{r'}}{\alpha _{r'}} > 0$. 
\item $a_1+\dots +a_i + \frac{(m_0-\alpha _1 - \dots - \alpha _i)a_{r'}}{\alpha _{r'}} + b > 0$ for every $i=1,\dots ,r'-1$. 
Moreover, $\frac{m_0a_{r'}}{\alpha _{r'}} + b > 0$. 
\end{enumerate}
\end{lem}
\begin{proof}
Consider the following divisors on $\wS$: 
\begin{align*}
\wDelta &:= \wD _0 + m_0\wF - \sum _{i=1}^{r'} \sum _{\lambda = 1}^{\alpha _i} \lambda \wD _{i,\lambda} + \sum _{\mu = m_0-\alpha '}^{\alpha _{r'}}(\mu - m_0 + \alpha ')\wD _{r',\mu} - \wE _{r+1}' ,\\
\wDelta' &:= \wD _0 + m_0\wF - \sum _{i=1}^{r'} \sum _{\lambda = 1}^{\alpha _i} \lambda \wD _{i,\lambda} + \sum _{\mu = m_0-\alpha '}^{\alpha _{r'}}(\mu - m_0 + \alpha ')\wD _{r',\mu} - \wE _{r+1} .\\
\end{align*}
Since $(\wDelta \cdot K_{\wS}-\wF) = -2 <0$, $(\wDelta)^2 = -1$ and $(\wDelta \cdot -K_{\wS}) = 1$, we obtain $\dim |\wDelta| \ge 0$ by the Riemann-Roch theorem  and rationality of $\wS$. 
Similarly, we also obtain $\dim |\wDelta '| \ge 0$. 

We only show the assertion (1). 
Indeed, the assertion (2) can be shown by the same argument replacing $\wDelta$ by $\wDelta '$ in (1). 

At first, we shall show the first assertion with $i=r'-1$. 
Since $(\wDelta \cdot \wD _0) = (\wDelta \cdot \wD _{\infty}) = (\wDelta \cdot \wD _{r+1,0}) = (\wDelta \cdot \wD _{r+1,1}) = \dots = (\wDelta \cdot \wD _{r+1,\beta - 1}) = 0$, we have: 
\begin{align*}
(E _i \cdot f_{\ast}(\wDelta)) = \left\{ \begin{array}{cc}
1 & \text{if}\ i <r' \\ \frac{m_0-\alpha '}{\alpha _{r'}} & \text{if}\ i =r' \\ 0 & \text{if}\ i >r'
\end{array} \right. 
\end{align*}
for $i=1,\dots ,r,r+1$. Hence: 
\begin{align*}
0 < (H \cdot f_{\ast}(\wDelta)) = \sum _{i=1}^ra_i(E_i \cdot f_{\ast}(\wDelta)) + b(E_{r+1} \cdot f_{\ast}(\wDelta)) = a_1 + \dots + a_{r'-1} + \frac{(m_0-\alpha ')a_{r'}}{\alpha _{r'}} . 
\end{align*}

In what follows, we shall show the first assertion for the general case. 
By using the above result combined with $\frac{a_{i+1}}{\alpha _{i+1}} \le \dots \le \frac{a_{r'-1}}{\alpha _{r'-1}} \le \frac{a_{r'}}{\alpha _{r'}}$, we have: 
\begin{align*}
0 &< a_1 + \dots + a_{r'-1} + \frac{(m_0-\alpha ')a_{r'}}{\alpha _{r'}} \\
&\le a_1 + \dots +a_i + \frac{(\alpha _{i+1} + \dots + \alpha _{r'-1} + m_0 - \alpha ')a_{r'}}{\alpha _{r'}} \\
&= a_1 + \dots +a_i + \frac{(m_0 - \alpha _1 -\dots - \alpha _i)a_{r'}}{\alpha _{r'}}. 
\end{align*}

We shall prove the remaining assertion.
Similarly, we have: 
\begin{align*}
0 < a_1 + \dots + a_{r'-1} + \frac{(m_0-\alpha ')a_{r'}}{\alpha _{r'}} \le \frac{m_0a_{r'}}{\alpha _{r'}}. 
\end{align*}
By virtue of $m_0 > 0$, we thus obtain $\frac{a_{r'}}{\alpha _{r'}} > 0$. 
\end{proof}
Then the following result holds: 
\begin{lem}\label{lem(3-6-2)}
If $\beta = 1$, then $H \in \Ampc (S)$. 
\end{lem}
\begin{proof}
Since $\beta = 1$, we note $F \sim _{\bQ} E_{r+1} + E_{r+1}'$ and $m_0F \sim _{\bQ} \sum _{i=1}^r\alpha _iE_i$ because $\wF \sim \wE _{r+1}' + \wD _{r+1,0} + \wE _{r+1}$ and $\wD _{\infty} \sim \wD _0 + m_0\wF - \sum _{i=1}^r \sum _{\lambda = 1}^{\alpha _i} \lambda \wD _{i,\lambda}$. 

Suppose that $\frac{a_1}{\alpha _1} = \frac{a_{r'}}{\alpha _{r'}}$. 
Note that $\frac{a_{r'}}{\alpha _{r'}} > 0$ and $b + \frac{m_0a_{r'}}{\alpha _{r'}} > 0$ by Lemma \ref{lem(3-6-1)}. 
Letting $\varepsilon$ be a positive rational number satisfying $\varepsilon < \frac{a_{r'}}{\alpha _{r'}}$, we take the effective $\bQ$-divisor: 
\begin{align*}
D := \sum _{i=1}^r\alpha _i \left( \frac{a_i}{\alpha _i} - \frac{a_{r'}}{\alpha _{r'}} + \varepsilon \right)E_i + \left\{ b + m_0 \left( \frac{a_{r'}}{\alpha _{r'}} - \varepsilon \right)\right\} E_{r+1} + m_0 \left( \frac{a_{r'}}{\alpha _{r'}} - \varepsilon \right)E_{r+1}'
\end{align*}
on $S$. 
Then we know $H \sim _{\bQ} D$ and: 
\begin{align*}
S \backslash \Supp (D) \simeq \wS \backslash \Supp \left( \wD _0 + \wD _{\infty} + \sum _{i=1}^r (\wF _i - \wE _i') + \wF _{r+1} \right) \simeq \bA ^1_{\bk} \times \bA ^1_{\ast , \bk}
\end{align*}
by Lemma \ref{lem(2-1-1)} (2). 
Thus, $H \in \Ampc (S)$. 

In what follows, we can assume that $\frac{a_1}{\alpha _1} < \frac{a_{r'}}{\alpha _{r'}}$. 
Put $r'' := \max \{ i \in \{ 1,\dots ,r'-1\} \,|\, \frac{a_i}{\alpha _i} < \frac{a_{r'}}{\alpha _{r'}} \}$. 
For simplicity, we put $\alpha '' := \sum _{i=1}^{r''}\alpha _i$ and $a'' := \sum _{i=1}^{r''}a_i$. 
Then we note $m_0 - \alpha '' > 0$. 
Moreover, by using Lemma \ref{lem(3-6-1)} we have: 
\begin{align*}
\frac{1}{m_0-\alpha''}\left( a'' +b + \frac{(m_0-\alpha'')a_{r'}}{\alpha _{r'}} \right) > 0. 
\end{align*}
and: 
\begin{align*}
\frac{1}{m_0-\alpha''}\left( a'' + \frac{(m_0-\alpha'')a_{r'}}{\alpha _{r'}} \right) > 0. 
\end{align*}
Letting $\varepsilon$ be a positive rational number satisfying: 
\begin{align*}
\varepsilon < \min \left\{ \frac{a_{r'}}{\alpha _{r'}} - \frac{a_1}{\alpha _1},\ \frac{1}{m_0-\alpha''}\left( a'' + b + \frac{(m_0-\alpha'')a_{r'}}{\alpha _{r'}} \right),\ \frac{1}{m_0-\alpha''}\left( a'' + \frac{(m_0-\alpha'')a_{r'}}{\alpha _{r'}} \right) \right\} ,
\end{align*}
we take the effective $\bQ$-divisor: 
\begin{align*}
D &:= \sum _{i=1}^{r''} \alpha _i \left( \frac{a_{r'}}{\alpha _{r'}} -\frac{a_i}{\alpha _i} - \varepsilon \right)E_i' + 
\sum _{j=r''+1}^r\alpha _j \left( \frac{a_j}{\alpha _j} - \frac{a_{r'}}{\alpha _{r'}} + \varepsilon \right)E_j \\
&\qquad + \left\{ a'' + b + (m_0 - \alpha '') \left( \frac{a_{r'}}{\alpha _{r'}} - \varepsilon \right) \right\} E_{r+1} 
+ \left\{ a'' + (m_0 - \alpha '') \left( \frac{a_{r'}}{\alpha _{r'}} - \varepsilon \right) \right\} E_{r+1}'
\end{align*}
on $S$. 
Then $H \sim _{\bQ} D$. 
Moreover, we know: 
\begin{align*}
S \backslash \Supp (D) \simeq \wS \backslash \Supp \left( \wD _0 + \wD _{\infty} + \sum _{i=1}^{r''} (\wF _i - \wE _i) + \sum _{j=r''+1}^r (\wF _j - \wE _j') + \wF _{r+1} \right) \simeq \bA ^1_{\bk} \times \bA ^1_{\ast , \bk}
\end{align*}
by Lemma \ref{lem(2-1-1)} (2). 
Thus, $H \in \Ampc (S)$. 
\end{proof}
From now on, we thus consider the case $\beta \ge 2$. Assume $m_0=2$ in what follows. 
Then we obtain the following lemma: 
\begin{lem}\label{lem(3-6-3)}
The following assertions hold: 
\begin{enumerate}
\item If $\alpha _1 > 1$, then $a_1>0$. 
\item If $\alpha > 1$ and $\alpha _1 = 1$, then $a_1+\frac{a_2}{\alpha _2}>0$. 
\item If $\beta \ge 3$, then $b>0$. 
\item If $\beta \ge 2$, then $\frac{(\beta -1)a_1}{\alpha _1} + b > 0$. 
\end{enumerate}
\end{lem}
\begin{proof}
In (1), consider the divisor $\wDelta := \wD _0 + 2\wF - \wD _{1,1} - \sum _{\lambda=2}^{\alpha _1}2\wD _{1,\lambda} - \wE _{r+1}'$. 
Since $(\wDelta \cdot K_{\wS}-\wF) = -2 <0$, $(\wDelta)^2 = -1$ and $(\wDelta \cdot -K_{\wS}) = 1$, we obtain $\dim |\wDelta| \ge 0$ by the Riemann-Roch theorem and rationality of $\wS$. 
Hence, we have $0 < (H \cdot f_{\ast}(\wDelta)) = \frac{2a_1}{\alpha _1}$. 
Namely, we obtain $a_1>0$. 

In (2), consider the divisor $\wDelta := \wD _0 + 2\wF - \wE _1 - \sum _{\lambda = 1}^{\alpha _2}\wD _{2,\lambda} - \wE _{r+1}'$. 
Since $(\wDelta \cdot K_{\wS}-\wF) = -2 <0$, $(\wDelta)^2 = -1$ and $(\wDelta \cdot -K_{\wS}) = 1$, we obtain $\dim |\wDelta| \ge 0$ by the Riemann-Roch theorem and rationality of $\wS$. 
Hence, we have $0 < (H \cdot f_{\ast}(\wDelta)) = a_1 + \frac{a_2}{\alpha _2}$. 

In (3), we first assume $\beta = 3$. 
Then consider the divisor $\wDelta := \wD _0 + 2\wF - \sum _{\mu = 1}^3\mu \wD_{r+1,\mu}$. 
Since $(\wDelta \cdot K_{\wS}-\wF) = -2 <0$, $(\wDelta)^2 = -1$ and $(\wDelta \cdot -K_{\wS}) = 1$, we obtain $\dim |\wDelta | \ge 0$ by the Riemann-Roch theorem and rationality of $\wS$. 
Hence, we have $0 < (H \cdot f_{\ast}(\wDelta)) = b$ because $(\wDelta \cdot \wD _0) = (\wDelta \cdot \wD _{\infty}) = 0$. 

In what follows, we can assume $\beta \ge 4$. 
Then consider the divisor $\wDelta := (\beta -1)\wD _0 + 2(\beta -1)\wF - \sum _{\mu =1}^{\beta} 2\mu \wD _{r+1,\mu} - (\beta -3)\wE _{r+1}'$. 
Since $(\wDelta \cdot K_{\wS} - \wF) = -2(\beta - 1) < 0$, $(\wDelta )^2 = \beta ^2 -2\beta - 7$ and $(\wDelta \cdot -K_{\wS}) = \beta - 1$, we obtain $\dim |\wDelta| \ge \frac{1}{2}(\beta ^2 - \beta - 8) \ge 0$ by the Riemann-Roch theorem and rationality of $\wS$. 
Hence, we have $0 < \frac{1}{2}(H \cdot f_{\ast}(\wDelta)) = b$ because $(\wDelta \cdot \wD _0) = (\wDelta \cdot \wD _{\infty}) = 0$. 

In (4), consider the divisor $\wDelta := (\beta -1)\wD _0 + 2(\beta -1)\wF - \sum _{\lambda = 1}^{\alpha _1}(\beta -1)\wD _{1,\lambda} - \sum _{\mu =1}^{\beta} \mu \wD _{r+1,\mu} - (\beta -2)\wE _{r+1}'$. 
Since $(\wDelta \cdot K_{\wS} - \wF) = -2(\beta - 1) < 0$, $(\wDelta )^2 = \beta - 3$ and $(\wDelta \cdot -K_{\wS}) = \beta - 1$, we obtain $\dim |\wDelta| \ge \beta - 2 \ge 0$ by the Riemann-Roch theorem. 
Hence, we have $0 < (H \cdot f_{\ast}(\wDelta)) = \frac{(\beta -1)a_1}{\alpha _1} + b$ because $(\wDelta \cdot \wD _0) = (\wDelta \cdot \wD _{\infty}) = 0$. 
\end{proof}
We note $E_{r+1}' \sim _{\bQ} F - E_{r+1}$ and $2F \sim _{\bQ} \sum _{i=1}^r\alpha _iE_i + (\beta - 1)E_{r+1}$ because $\wF \sim \wE _{r+1}' + \sum _{\mu =0}^{\beta}\wD_{r+1,\mu}$ and $\wD _{\infty} \sim \wD _0 + 2\wF - \sum _{i=1}^r \sum _{\lambda = 1}^{\alpha _i}\lambda \wD _{i,\lambda} - \sum _{\mu = 1}^{\beta}\mu \wD _{r+1,\mu} + \wE _{r+1}$. 
Then the following result holds: 
\begin{lem}\label{lem(3-6-4)}
If $\beta \ge 2$ and $m_0=2$, then $H \in \Ampc (S)$. 
\end{lem}
\begin{proof}
We consider the following two cases separately. 
\smallskip

\noindent
{\bf Case 1:} ($a_1>0$). 
In this case, we note $\frac{a_1}{\alpha _1} > 0$. 
We consider the following two subcases separately. 
\smallskip

\noindent
{\bf Subcase 1-1:} ($\beta = 2$). 
In this subcase, we note $\frac{a_1}{\alpha _1}+b>0$ by Lemma \ref{lem(3-6-3)} (4). 
Letting $\varepsilon$ be a positive rational number satisfying $\varepsilon < \min \left\{ \frac{a_1}{\alpha _1} + b, \frac{a_1}{\alpha _1} \right\}$, we take the effective $\bQ$-divisor: 
\begin{align*}
D := \sum _{i=1}^r \alpha _i \left( \frac{a_i}{\alpha _i} - \frac{a_1}{\alpha _1} + \varepsilon \right) E_i + \left( \frac{a_1}{\alpha _1} + b -\varepsilon \right) E_{r+1} + 2 \left( \frac{a_1}{\alpha _1} - \varepsilon \right) E_{r+1}'
\end{align*}
on $S$. 
Then we know $H \sim _{\bQ} D$ and: 
\begin{align*}
S \backslash \Supp (D) \simeq \wS \backslash \Supp \left( \wD _0 + \wD _{\infty} + \sum _{i=1}^r (\wF _i - \wE _i') + \wF _{r+1} \right) \simeq \bA ^1_{\bk} \times \bA ^1_{\ast , \bk}
\end{align*}
by Lemma \ref{lem(2-1-1)} (2). 
Thus, $H \in \Ampc (S)$. 
\smallskip

\noindent
{\bf Subcase 1-2:} ($\beta \ge 3$). 
In this subcase, we note $b>0$ by Lemma \ref{lem(3-6-3)} (3). 
Letting $\varepsilon$ be a positive rational number satisfying:
\begin{align*}
\varepsilon <
\left\{ \begin{array}{ll}
\frac{a_1}{\alpha _1} & \text{if}\ \beta = 3\\
\min \left\{ \frac{a_1}{\alpha _1},\ \frac{b}{\beta-3} \right\} & \text{if}\ \beta > 3
\end{array}\right. ,
\end{align*}
we take the effective $\bQ$-divisor: 
\begin{align*}
D := \sum _{i=1}^r \alpha _i \left( \frac{a_i}{\alpha _i} - \varepsilon \right) E_i + \{b - (\beta -3)\varepsilon \} E_{r+1} + 2 \varepsilon E_{r+1}'
\end{align*}
on $S$. 
Then we know $H \sim _{\bQ} D$ and: 
\begin{align*}
S \backslash \Supp (D) \simeq \wS \backslash \Supp \left( \wD _0 + \wD _{\infty} + \sum _{i=1}^r (\wF _i - \wE _i') + \wF _{r+1} \right) \simeq \bA ^1_{\bk} \times \bA ^1_{\ast , \bk}
\end{align*}
by Lemma \ref{lem(2-1-1)} (2). 
Thus, $H \in \Ampc (S)$. 
\smallskip

\noindent
{\bf Case 2:} ($a_1 \le 0$). 
In this case, we notice $\alpha _1 = 1$ by Lemma \ref{lem(3-6-3)} (1). 
In particular, $(\beta -1)a_1+b>0$ by using Lemma \ref{lem(3-6-3)} (4). 
We consider the following two subcases separately. 
\smallskip

\noindent
{\bf Subcase 2-1:} ($\alpha = 1$). 
In this subcase, we notice $r=1$ and $\beta \ge 4$. 
Letting $\varepsilon$ be a positive rational number satisfying $\varepsilon < \frac{(\beta - 1)a_1+b}{\beta -2}$, we take the effective $\bQ$-divisor: 
\begin{align*}
D := (-2a_1+\varepsilon)E_1' + \{(\beta -1)a_1+b - (\beta -2) \varepsilon \}E_2 + \varepsilon E_2' 
\end{align*}
on $S$. 
Then we know $H \sim _{\bQ} D$ and: 
\begin{align*}
S \backslash \Supp (D) \simeq \wS \backslash \Supp \left( \wD _0 + \wD _{\infty} + \wE _1' + \wF _2 \right) \simeq \bA ^1_{\bk} \times \bA ^1_{\ast , \bk}
\end{align*}
by Lemma \ref{lem(2-1-1)} (2). 
Thus, $H \in \Ampc (S)$. 
\smallskip

\noindent
{\bf Subcase 2-2:} ($\alpha > 1$). 
In this subcase, we have $a_1+\frac{a_2}{\alpha _2} > 0$ by Lemma \ref{lem(3-6-3)} (2). 
Letting $\varepsilon$ be a positive rational number satisfying:
\begin{align*}
\varepsilon <
\left\{ \begin{array}{ll}
a_1+\frac{a_2}{\alpha _2} & \text{if}\ \beta = 2 \\
\min \left\{ a_1+\frac{a_2}{\alpha _2},\ \frac{(\beta - 1)a_1+b}{\beta -2} \right\} & \text{if}\ \beta > 2
\end{array}\right. ,
\end{align*}
we take the effective $\bQ$-divisor: 
\begin{align*}
D := (-2a_1+\varepsilon)E_1' + \sum _{i=2}^r \alpha _i \left( a_1 + \frac{a_i}{\alpha _i} - \varepsilon \right) E_i + \{(\beta -1)a_1+b - (\beta -2) \varepsilon \}E_{r+1} + \varepsilon E_{r+1}' 
\end{align*}
on $S$. 
Then we know $H \sim _{\bQ} D$ and: 
\begin{align*}
S \backslash \Supp (D) \simeq \wS \backslash \Supp \left( \wD _0 + \wD _{\infty} + \wE _1' + \sum _{i=2}^r (\wF _i - \wE _i') + \wF _{r+1} \right) \simeq \bA ^1_{\bk} \times \bA ^1_{\ast , \bk}
\end{align*}
by Lemma \ref{lem(2-1-1)} (2). 
Thus, $H \in \Ampc (S)$. 
\smallskip

Lemma \ref{lem(3-6-4)} is thus verified. 
\end{proof}
Proposition \ref{prop(3-6)} follows from Lemmas \ref{lem(3-6-2)} and \ref{lem(3-6-4)}. 
The proof of Proposition \ref{prop(3-6)} is thus completed. 
\subsection{Type $\sA_2$ case}\label{4-7}
In this subsection, we keep the notation from \S \S \ref{4-1} and assume further that $g$ satisfies $(\ast \ast)$, $(s,t) = (0,0)$ and $(\wD_0 \cdot \wD_{\infty})  = 1$. 
Namely, $-m_{\infty} = (m_0+2) - \alpha$. 
Note that the configuration of the $\bP ^1$-fibration $g:\wS \to \bP ^1_{\bk}$ is given as in Figure \ref{Case5}. 
\begin{figure}\begin{center}\scalebox{0.7}
{\begin{tikzpicture}
\draw [very thick] (-1,0)--(7,0);
\node at (-1.4,0) {$\wD _0$};
\node at (4.25,3) {\Large $\cdots \cdots$};
\draw [very thick] (1.25,5.5)--(7,5.5);
\draw [very thick] (-0.5,-0.25)--(1.25,5.5);
\node at (0.6,5) {$\wD _{\infty}$};
\node at (7.6,0) {$-m_0$};
\node at (7.6,5.5) {$-m_{\infty}$};

\draw [dashed] (3,-0.25)--(2.5,1);
\draw (2.5,0.75)--(3,2);
\draw (3,1.75)--(2.5,3);
\draw (2.5,3.75)--(3,5);
\draw [dashed] (3,4.75)--(2.5,6);
\node at (2.5,3.5) {$\vdots$};
\node at (3,-0.6) {\footnotesize $\wE _1'$};
\node at (3.2,1.25) {\footnotesize $\wD _{1,1}$};
\node at (2.4,2.25) {\footnotesize $\wD _{1,2}$};
\node at (2,4.3) {\footnotesize $\wD _{1,\alpha _1-1}$};
\node at (2.5,6.3) {\footnotesize $\wE _1$};

\draw [dashed] (6.54,-0.25)--(6,1);
\draw (6,0.75)--(6.5,2);
\draw (6.5,1.75)--(6,3);
\draw (6,3.75)--(6.5,5);
\draw [dashed] (6.5,4.75)--(6,6);
\node at (6,3.5) {$\vdots$};
\node at (6.5,-0.6) {\footnotesize $\wE _r'$};
\node at (6.7,1.25) {\footnotesize $\wD _{r,1}$};
\node at (5.9,2.25) {\footnotesize $\wD _{r,2}$};
\node at (5.5,4.3) {\footnotesize $\wD _{r,\alpha _r-1}$};
\node at (6,6.3) {\footnotesize $\wE _r$};
\end{tikzpicture}}
\caption{The configuration of $g : \wS \to \bP ^1_{\bk}$ in Subsection \ref{4-7}.}\label{Case5}
\end{center}\end{figure}
Then we have the following result, which is the main result of this subsection: 
\begin{prop}\label{prop(3-5)}
With the same notations and assumptions as above, we have $\Ampc (S) = \Amp (S)$. 
\end{prop}
In what follows, we shall prove the above result. 
Let $H \in \Amp (S)$. 
Since $\Amp (S)$ is contained in $\Cl (S)_{\bQ} = \bigoplus _{i=1}^r \bQ [E_i]$, we can write: 
\begin{align*}
H \sim _{\bQ} \sum _{i=1}^r a_iE_i
\end{align*}
for some rational numbers $a_1,\dots ,a_r$. 
Without loss of generality, we may assume $\frac{a_1}{\alpha _1} \le \frac{a_2}{\alpha _2} \le \dots \le \frac{a_r}{\alpha _r}$. 
Since $\alpha - 2 = m_0+m_{\infty}$ and $m_{\infty}>0$, we have $\alpha _1 + \dots + \alpha _r \ge m_0+1$. 
We set $r' := \min \{ i \in \{ 1,\dots ,r\} \,|\, \alpha _1 + \dots + \alpha _i \ge m_0+1\}$. 
For simplicity, we put $\alpha ' := \sum _{i=1}^{r'-1}\alpha _i$, where $\alpha ' = 0$ provided $r'=1$. 
Then the following lemma holds: 
\begin{lem}\label{lem(3-5-1)}
$a_1+\dots +a_i + \frac{(m_0 + 1 -\alpha _1 - \dots - \alpha _i)a_{r'}}{\alpha _{r'}} > 0$ for every $i=1,\dots ,r'-1$. 
Moreover, $\frac{a_{r'}}{\alpha _{r'}} > 0$. 
\end{lem}
\begin{proof}
Consider the following divisor on $\wS$: 
\begin{align*}
\wDelta := \wD _0 + m_0\wF - \sum _{i=1}^{r'} \sum _{\lambda = 1}^{\alpha _i} \lambda \wD _{i,\lambda} + \sum _{\mu = m_0-\alpha '+1}^{\alpha _{r'}}(\mu - m_0 -1 + \alpha ')\wD _{r',\mu}. 
\end{align*}
Since $(\wDelta \cdot K_{\wS}-\wF) = -2 <0$, $(\wDelta)^2 = -1$ and $(\wDelta \cdot -K_{\wS}) = 1$, we obtain $\dim |\wDelta| \ge 0$ by the Riemann-Roch theorem and rationality of $\wS$. 

At first, we shall show the first assertion with $i=r'-1$. 
Since $(\wDelta \cdot \wD _0) = (\wDelta \cdot \wD _{\infty}) = 0$, we have: 
\begin{align*}
(E _i \cdot f_{\ast}(\wDelta)) = \left\{ \begin{array}{cc}
1 & \text{if}\ i <r' \\ \frac{m_0+1-\alpha '}{\alpha _{r'}} & \text{if}\ i =r' \\ 0 & \text{if}\ i >r'
\end{array} \right. 
\end{align*}
for $i=1,\dots ,r$. Hence: 
\begin{align*}
0 < (H \cdot f_{\ast}(\wDelta)) = \sum _{i=1}^ra_i(E_i \cdot f_{\ast}(\wDelta)) = a_1 + \dots + a_{r'-1} + \frac{(m_0+1-\alpha ')a_{r'}}{\alpha _{r'}}. 
\end{align*}

In what follows, we shall show the first assertion for the general case. 
By using the above result combined with $\frac{a_{i+1}}{\alpha _{i+1}} \le \dots \le \frac{a_{r'-1}}{\alpha _{r'-1}} \le \frac{a_{r'}}{\alpha _{r'}}$, we have: 
\begin{align*}
0 &< a_1+\dots +a_{r'-1} + \frac{(m_0+1-\alpha')a_{r'}}{\alpha _{r'}} \\
&\le a_1 + \dots +a_i + \frac{(\alpha _{i+1} + \dots + \alpha _{r'-1} + m_0 + 1 - \alpha ')a_{r'}}{\alpha _{r'}} \\
&= a_1 + \dots +a_i + \frac{(m_0 + 1 - \alpha _1 - \dots - \alpha _i)a_{r'}}{\alpha _{r'}}. 
\end{align*}

We shall prove the remaining assertion.
Similarly, we have: 
\begin{align*}
0 < a_1+\dots +a_{r'-1} + \frac{(m_0+1-\alpha')a_{r'}}{\alpha _{r'}} \le \frac{(m_0+1)a_{r'}}{\alpha _{r'}}. 
\end{align*}
By virtue of $m_0+1 \ge 3$, we thus obtain $\frac{a_{r'}}{\alpha _{r'}} > 0$. 
\end{proof}
Note $(m_0+1)F \sim _{\bQ} \sum _{i=1}^r \alpha _i E_i$ because $\wD _{\infty} \sim \wD _0 + (m_0+1)\wF - \sum _{i=1}^r \sum _{\lambda = 1}^{\alpha _i} \lambda \wD _{i,\lambda}$. 
Since $(\wD _0 \cdot \wD _{\infty}) = 1$, there exists a unique fiber $\wF _0$ of $g$ such that $\wD _0 \cap \wD _{\infty} \cap \wF _0 \not= \emptyset$. 
Put $F_0 := f_{\ast}(\wF _0)$. 

Assume $\frac{a_1}{\alpha _1} = \frac{a_{r'}}{\alpha _{r'}}$. 
Note that $\frac{a_{r'}}{\alpha _{r'}} > 0$ by Lemma \ref{lem(3-5-1)}. 
Letting $\varepsilon$ be a positive rational number satisfying $\varepsilon < \frac{a_{r'}}{\alpha _{r'}}$, we take the effective $\bQ$-divisor: 
\begin{align*}
D := (m_0+1) \left( \frac{a_{r'}}{\alpha _{r'}} - \varepsilon \right)F_0 + \sum _{i=1}^r\alpha _i \left( \frac{a_i}{\alpha _i} - \frac{a_{r'}}{\alpha _{r'}} + \varepsilon \right)E_i
\end{align*}
on $S$. 
Then we know $H \sim _{\bQ} D$ and: 
\begin{align*}
S \backslash \Supp (D) \simeq \wS \backslash \Supp \left( \wD _0 + \wD _{\infty} + \wF _0 + \sum _{i=1}^r (\wF _i - \wE _i') \right) \simeq \bA ^1_{\bk} \times \bA ^1_{\ast , \bk}
\end{align*}
by Lemma \ref{lem(2-1-1)} (3). 
Thus, $H \in \Ampc (S)$. 

In what follows, we can assume $\frac{a_1}{\alpha _1} < \frac{a_{r'}}{\alpha _{r'}}$. 
Put $r'' := \max \{ i \in \{ 1,\dots ,r'-1\} \,|\, \frac{a_i}{\alpha _i} < \frac{a_{r'}}{\alpha _{r'}} \}$. 
For simplicity, we put $\alpha '' := \sum _{i=1}^{r''}\alpha _i$ and $a'' := \sum _{i=1}^{r''}a_i$. 
Then we note $m_0 + 1 - \alpha '' > 0$. 
Moreover, by Lemma \ref{lem(3-5-1)} we have: 
\begin{align*}
\frac{1}{m_0+1-\alpha''}\left( a'' + \frac{(m_0+1-\alpha'')a_{r'}}{\alpha _{r'}} \right) > 0. 
\end{align*}
Letting $\varepsilon$ be a positive rational number satisfying: 
\begin{align*}
\varepsilon < \min \left\{ \frac{a_{r'}}{\alpha _{r'}} - \frac{a_1}{\alpha _1},\ \frac{1}{m_0+1-\alpha''}\left( a'' + \frac{(m_0+1-\alpha'')a_{r'}}{\alpha _{r'}} \right) \right\} , 
\end{align*}
we take the effective $\bQ$-divisor: 
\begin{align*}
D &:= \left\{ a'' + (m_0+1-\alpha '') \left( \frac{a_{r'}}{\alpha _{r'}}-\varepsilon \right) \right\} F_0 \\
&\qquad + \sum _{i=1}^{r''} \alpha _i \left( \frac{a_{r'}}{\alpha _{r'}} -\frac{a_i}{\alpha _i} - \varepsilon \right)E_i' + 
\sum _{j=r''+1}^r\alpha _j \left( \frac{a_j}{\alpha _j} - \frac{a_{r'}}{\alpha _{r'}} + \varepsilon \right)E_j 
\end{align*}
on $S$. 
Then $H \sim _{\bQ} D$. 
Moreover, we know: 
\begin{align*}
S \backslash \Supp (D) \simeq \wS \backslash \Supp \left( \wD _0 + \wD _{\infty} + \wF _0 + \sum _{i=1}^{r''} (\wF _i - \wE _i) + \sum _{j=r''+1}^r (\wF _j - \wE _j') \right) \simeq \bA ^1_{\bk} \times \bA ^1_{\ast , \bk}
\end{align*}
by Lemma \ref{lem(2-1-1)} (3). 
Thus, $H \in \Ampc (S)$. 

The proof of Proposition \ref{prop(3-5)} is thus completed. 
\section{Proof of Theorem \ref{main(1)}}\label{5}
In this section, we shall prove Theorem \ref{main(1)}. 
Let $S$ be a Du Val del Pezzo surface of degree $d \ge 2$. 
If $\Ampc (S) = \Amp (S)$, then we obtain that $-K_S \in \Ampc (S)$ because $-K_S$ is ample. 

From now on, we assume that $-K_S \in \Ampc (S)$. 
If $d \ge 3$, then we obtain that $\Ampc (S) = \Amp (S)$ by {\cite{Saw25}}. 
Hence, we can assume further that $d=2$. 
If $\rho (S) = 1$, then we know that $\Ampc (S) = \Amp (S) = \bQ _{>0}[-K_S]$. 
In what follows, we thus assume that $\rho (S) \ge 2$. 
Let $f:\wS \to S$ be the minimal resolution, and let $\wD$ be the reduced exceptional divisor of $f$. 
\begin{lem}\label{lem(4-1)}
If $S$ has a non-cyclic quotient singular point, which is not of type $\sD _5$, then $\Ampc (S) = \Amp (S)$. 
\end{lem}
\begin{proof}
By assumption and $\rho (S) \ge 2$, ${\rm Dyn}(S) = \sD _4$, $\sD _4+\sA _1$, $\sD_ 4+2\sA _1$, $\sD _6$ or $\sE _6$. 
Indeed, it can be seen from the classification of Du Val del Pezzo surfaces of degree $2$ (see, e.g., {\cite{Ura81}} or {\cite[\S 8.7]{Dol12}}). 
We consider this proof according to Dynkin types of $S$.  

In ${\rm Dyn}(S) = \sD _4$, by the configuration of this Du Val del Pezzo surface (cf.\ {\cite[p.\ 1220]{CPW16b}}), we can find a $\bP ^1$-fibration $g:\wS \to \bP ^1_{\bk}$ such that there exists exactly one $(-2)$-curve is a section of $g$ and other $(-2)$-curves are fiber components of $g$ (see Figure \ref{fig(4-1)} (1)). 
In particular, $g$ satisfies $(\ast )$. 
Moreover, we have $(\alpha , \beta ,\beta' ,\gamma) = (0,3,3,0)$ provided that we use notations from \S \S \ref{4-1}. 

In ${\rm Dyn}(S) = \sD _4 + \sA _1$, by the configuration of this Du Val del Pezzo surface (cf.\ {\cite[p.\ 1217]{CPW16b}}), we can find a $\bP ^1$-fibration $g:\wS \to \bP ^1_{\bk}$ such that there exists exactly one $(-2)$-curve is a section of $g$ and other $(-2)$-curves are fiber components of $g$ (see Figure \ref{fig(4-1)} (2)). 
In particular, $g$ satisfies $(\ast )$. 
Moreover, we have $(\alpha , \beta ,\beta' ,\gamma) = (0,2,2,2)$ provided that we use notations from \S \S \ref{4-1}. 

In ${\rm Dyn}(S) = \sD _4 + 2\sA _1$, by the configuration of this Du Val del Pezzo surface (cf.\ {\cite[p.\ 1217]{CPW16b}}), we can find a $\bP ^1$-fibration $g:\wS \to \bP ^1_{\bk}$ such that there exists exactly one $(-2)$-curve is a section of $g$ and other $(-2)$-curves are fiber components of $g$ (see Figure \ref{fig(4-1)} (3)). 
In particular, $g$ satisfies $(\ast )$. 
Moreover, we have $(\alpha , \beta ,\beta' ,\gamma) = (0,1,1,4)$ provided that we use notations from \S \S \ref{4-1}. 

In ${\rm Dyn}(S) = \sD _6$, by the configuration of this Du Val del Pezzo surface (cf.\ {\cite[p.\ 1216]{CPW16b}}), we can find a $\bP ^1$-fibration $g:\wS \to \bP ^1_{\bk}$ such that there exists exactly one $(-2)$-curve is a section of $g$ and other $(-2)$-curves are fiber components of $g$ (see Figure \ref{fig(4-1)} (4)). 
In particular, $g$ satisfies $(\ast )$. 
Moreover, we have $(\alpha , \beta ,\beta' ,\gamma) = (0,1,1,4)$ provided that we use notations from \S \S \ref{4-1}. 

In ${\rm Dyn}(S) = \sE _6$, by the configuration of this Du Val del Pezzo surface (cf.\ {\cite[p.\ 1216]{CPW16b}}), we can find a $\bP ^1$-fibration $g:\wS \to \bP ^1_{\bk}$ such that there exists exactly one $(-2)$-curve is a section of $g$ and other $(-2)$-curves are fiber components of $g$ (see Figure \ref{fig(4-1)} (5)). 
In particular, $g$ satisfies $(\ast )$. 
Moreover, we have $(\alpha , \beta ,\beta' ,\gamma) = (1,0,0,5)$ provided that we use notations from \S \S \ref{4-1}. 

Hence, in every case, we obtain $\Amp (S) = \Ampc (S)$ by Proposition \ref{prop(4-1)}. 
\end{proof}
\begin{figure}
\begin{minipage}[c]{1\hsize}\begin{center}\scalebox{0.6}{\begin{tikzpicture}
\node at (-1,6) {\Large (1)};
\node at (-1,-2) {\ };
\node at (0.5,0) {\Large $\ast$};
\draw [thick] (0,0)--(5,0);
\draw [thick] (1,-0.5)--(1,5);
\draw [thick] (4,-0.5)--(4,5);
\draw [thick] (2.5,-0.5)--(2.5,5);
\draw [thick,dashed] (0.5,4)--(1.5,4);
\draw [thick,dashed] (2,4)--(3,4);
\draw [thick,dashed] (3.5,4)--(4.5,4);
\draw [thick,dashed] (0.5,2)--(1.5,2);
\draw [thick,dashed] (2,2)--(3,2);
\draw [thick,dashed] (3.5,2)--(4.5,2);
\end{tikzpicture}
\qquad \qquad
\begin{tikzpicture}
\node at (-1,6) {\Large (2)};
\node at (-1,-2) {\ };
\draw [thick] (0,0)--(5,0);
\node at (0.5,0) {\Large $\ast$};
\draw [thick] (1,-0.5)--(1,5);
\draw [thick] (2.5,-0.5)--(2.5,5);
\draw [thick] (3.7,-0.5)--(4.3,1.5);
\draw [thick] (3.7,5)--(4.3,3);
\draw [thick,dashed] (4.15,0.75)--(4.15,3.75);
\draw [thick,dashed] (0.5,4)--(1.5,4);
\draw [thick,dashed] (2,4)--(3,4);
\draw [thick,dashed] (0.5,2)--(1.5,2);
\draw [thick,dashed] (2,2)--(3,2);
\end{tikzpicture}
\qquad \qquad
\begin{tikzpicture}
\node at (-1,6) {\Large (3)};
\node at (-1,-2) {\ };
\draw [thick] (0,0)--(5,0);
\node at (0.5,0) {\Large $\ast$};
\draw [thick] (1,-0.5)--(1,5);
\draw [thick] (3.7,-0.5)--(4.3,1.5);
\draw [thick] (3.7,5)--(4.3,3);
\draw [thick,dashed] (4.15,0.75)--(4.15,3.75);
\draw [thick] (2.2,-0.5)--(2.8,1.5);
\draw [thick] (2.2,5)--(2.8,3);
\draw [thick,dashed] (2.65,0.75)--(2.65,3.75);
\draw [thick,dashed] (0.5,4)--(1.5,4);
\draw [thick,dashed] (0.5,2)--(1.5,2);
\end{tikzpicture}}\end{center}\end{minipage}

\begin{minipage}[c]{1\hsize}\begin{center}\scalebox{0.6}{\begin{tikzpicture}
\node at (-1,6) {\Large (4)};
\node at (-1,-2) {\ };
\draw [thick] (0,0)--(7,0);
\node at (0.5,0) {\Large $\ast$};
\draw [thick] (1,-0.5)--(1,5);
\draw [thick] (5.5,-0.5)--(6.3,1.5);
\draw [thick] (5.5,5)--(6.3,3);
\draw [thick] (6.15,0.75)--(6.15,3.75);
\draw [thick,dashed] (0,4)--(2,4);
\draw [thick] (4,2.25)--(7,2.25);
\draw [thick,dashed] (5,0.75)--(5,3.75);
\draw [thick,dashed] (0,2)--(2,2);
\end{tikzpicture}
\qquad \qquad
\begin{tikzpicture}
\node at (-1,6) {\Large (5)};
\node at (-1,-2) {\ };
\draw [thick] (0,0)--(7,0);
\node at (0.5,0) {\Large $\ast$};
\draw [thick,dashed] (1.3,-0.5)--(0.7,3);
\draw [thick,dashed] (0.7,2)--(1.3,5);
\draw [thick] (5.5,-0.5)--(6.3,1.5);
\draw [thick] (5.5,5)--(6.3,3);
\draw [thick] (6.15,0.75)--(6.15,3.75);
\draw [thick] (4.75,2)--(7,2.5);
\draw [thick] (5.75,2)--(3.25,2.5);
\draw [thick,dashed] (1.75,2)--(4.25,2.5);
\end{tikzpicture}}
\caption{Lemma \ref{lem(4-1)}; A dotted line (resp. a solid line) stands for a $(-1)$-curve (resp. a $(-2)$-curve); A line with $\ast$ means a non-fiber component of the $\bP ^1$-fibrations from $\wS$. }\label{fig(4-1)}
\end{center}\end{minipage}
\end{figure}
\begin{figure}
\begin{minipage}[c]{1\hsize}\begin{center}\scalebox{0.6}{\begin{tikzpicture}
\node at (-1,6.5) {\large (a)};
\node at (-1,-2) {\ };
\node at (0.5,0) {\Large $\ast$};
\draw [thick] (0,0)--(5,0);
\node at (-0.4,0) {$\wD _5$};
\draw [thick] (2,-0.25)--(4,2);
\draw [thick] (3.5,1)--(3.5,5);
\draw [thick] (4,3.75)--(2,6);
\node at (1.75,-0.6) {$\wD _4$};
\node at (1.75,6.3) {$\wD _2$};
\node at (3.5,5.5) {$\wD _3$};
\draw [thick] (0,3)--(4,3);
\draw [thick,dashed] (1,1.5)--(1,4.5);
\node at (-0.4,3) {$\wD _1$};
\node at (1,5) {$\wE _1$};
\end{tikzpicture}
\qquad \qquad
\begin{tikzpicture}
\node at (-1,6.5) {\large (b)};
\node at (-1,-2) {\ };
\node at (0.5,0) {\Large $\ast$};
\node at (0.5,5.75) {\Large $\ast$};
\draw [thick] (0,5.75)--(5,5.75);
\draw [thick] (0,0)--(5,0);
\node at (-0.4,5.75) {$\wD _1$};
\node at (-0.4,0) {$\wD _5$};
\draw [thick] (2,-0.25)--(4,2);
\draw [thick] (3.5,1)--(3.5,5);
\draw [thick] (4,3.75)--(2,6);
\node at (1.75,-0.6) {$\wD _4$};
\node at (1.75,6.3) {$\wD _2$};
\node at (4,5) {$\wD _3$};
\draw [thick,dashed] (0,3)--(4,3);
\node at (-0.4,3) {$\wE _3$};
\end{tikzpicture}
\qquad \qquad
\begin{tikzpicture}
\node at (-1,6.5) {\large (c)};
\node at (-1,-2) {\ };
\node at (0.5,0) {\Large $\ast$};
\node at (0.5,5.75) {\Large $\ast$};
\draw [thick] (0,5.75)--(1,5.75);
\draw [thick] (2,0.5)--(5,0.5);
\draw [very thick] (1,5.75) .. controls (1.5,5.75) and (1.5,0.5) .. (2,0.5);
\draw [thick] (0,0)--(5,0);
\node at (-0.4,5.75) {$\wD _1$};
\node at (-0.4,0) {$\wD _3$};
\draw [thick] (2.5,-0.25)--(4,2);
\draw [thick,dashed] (3.5,1)--(3.5,5);
\draw [thick] (4,3.75)--(2.5,6);
\node at (2.25,-0.6) {$\wD _2$};
\node at (2.25,6.3) {$\wD _4$};
\node at (4,5) {$\wE _2$};
\end{tikzpicture}}\end{center}\end{minipage}

\begin{minipage}[c]{1\hsize}\begin{center}\scalebox{0.6}{\begin{tikzpicture}
\node at (-1,6.5) {\large (d)};
\node at (0.5,0) {\Large $\ast$};
\node at (0.5,5.75) {\Large $\ast$};
\draw [thick] (0,5.75)--(5,5.75);
\draw [thick] (0,0)--(5,0);
\node at (-0.4,5.75) {$\wD _1$};
\node at (-0.4,0) {$\wD _3$};
\draw [thick] (2.5,-0.5)--(2.5,6.25);
\node at (3,-0.6) {$\wD _2$};
\draw [thick,dashed] (0,2)--(4,2);
\draw [thick,dashed] (0,4)--(4,4);
\node at (-0.4,2) {$\wE _2'$};
\node at (-0.4,4) {$\wE _2$};
\end{tikzpicture}
\qquad \qquad
\begin{tikzpicture}
\node at (-1,6.5) {\large (e)};
\node at (0.5,0) {\Large $\ast$};
\node at (0.5,5.75) {\Large $\ast$};
\draw [thick] (0,5.75)--(5,5.75);
\draw [thick] (0,0)--(5,0);
\node at (-0.4,5.75) {$\wD _1$};
\node at (-0.4,0) {$\wD _n$};
\draw [thick] (2.5,-0.25)--(4,1.5);
\draw [thick,dashed] (1.25,2)--(3.5,0.5);
\draw [thick] (2.5,2)--(4,1);
\draw [thick,dashed] (1.25,3.75)--(3.5,5.25);
\draw [thick] (4,4.25)--(2.5,6);
\draw [thick] (4,4.75)--(2.5,3.75);
\node at (2.25,-0.6) {$\wD _{n-1}$};
\node at (2.25,6.3) {$\wD _2$};
\node at (0.75,3.75) {$\wE _2$};
\node at (0.75,2) {$\wE _{n-1}$};
\node at (2.5,3.25) {$\vdots$};
\node at (2.5,2.75) {$\vdots$};
\end{tikzpicture}
\qquad \qquad
\begin{tikzpicture}
\node at (-1,6.5) {\large (f)};
\node at (4.5,0) {\Large $\ast$};
\node at (4.5,5.75) {\Large $\ast$};
\draw [thick] (3,5.75)--(5,5.75);
\draw [thick] (0.5,-0.25)--(1.5,4.75);
\draw [very thick] (1.5,4.75) .. controls (1.5,5) and (2,5.75) .. (3,5.75);
\draw [thick] (0,0)--(5,0);
\node at (0.5,-0.6) {$\wD _1$};
\node at (-0.4,0) {$\wD _2$};
\end{tikzpicture}}
\caption{Lemmas \ref{lem(4-2)}, \ref{lem(4-3)}, \ref{lem(4-4)}, \ref{lem(4-5)} and \ref{lem(4-6)}; Configurations of some sections and fiber components of $g$; A line with $\ast$ means a non-fiber component of $g$. }\label{fig(4-2)}
\end{center}\end{minipage}
\end{figure}
\begin{lem}\label{lem(4-2)}
If $S$ has a singular point of type $\sD _5$, then $\Ampc (S) = \Amp (S)$. 
\end{lem}
\begin{proof}
By assumption and $\rho (S) \ge 2$, ${\rm Dyn}(S) = \sD _5$ or $\sD _5+\sA _1$. 
Indeed, it can be seen from the classification of Du Val del Pezzo surfaces of degree $2$ (see, e.g., {\cite{Ura81}} or {\cite[\S 8.7]{Dol12}}). 
Let $P$ be a singular point on $S$ of type $\sD _5$ and let $\wD _1+\dots + \wD _5$ be the connected component of $\wD$ corresponding to $P$ such that the weighted dual graph of $\wD _1+\dots + \wD _5$ is given as follows: 
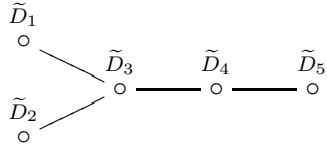
\begin{align*}
\xygraph{\circ ([]!{+(0,.25)} {^{\wD_5}}) -[l] \circ ([]!{+(0,.25)} {^{\wD_4}}) -[l] \circ ([]!{+(0,.25)} {^{\wD_3}}) (-[]!{+(-1,-0.5)} \circ ([]!{+(0,.25)} {^{\wD_2}}), -[]!{+(-1,0.5)} \circ ([]!{+(0,.25)} {^{\wD_1}}))}
\end{align*}
By using the list of the configurations of all $(-2)$-curves and some $(-1)$-curves on weak del Pezzo surfaces of degree $2$ (see {\cite[p.\ 20]{PW10}}), we know that there exists a $(-1)$-curve $\wE _1$ on $\wS$ such that $(\wE _1 \cdot \wD _i) = \delta _{1,i}$ for $i=1,\dots ,5$. 
Then a divisor $\wF := \wD_2+2(\wE_1+\wD_1+\wD_3)+\wD_4$ defines a $\bP ^1$-fibration $g := \Phi _{|\wF|}: \wS \to \bP ^1_{\bk}$, $\wD_5$ becomes a section of $g$ and each component of $\wD-\wD_5$ is a fiber component of $g$. 
Note that the configuration of $\wF + \wD _5$ looks like that in Figure \ref{fig(4-2)} (a). 
Moreover, $g$ satisfies $(\ast )$, $(s,t)=(0,1)$ and $\gamma = 4$ provided that we use notations from \S \S \ref{4-1}. 
Hence, we obtain that $\Amp (S) = \Ampc (S)$ by Proposition \ref{prop(4-2)}. 
\end{proof}
In what follows, we assume that $S$ has only cyclic quotient singular points. 
Thus, $S$ allows only singular points of types $\sA _1$, $\sA _2$, $\sA _3$, $\sA _4$, $\sA _5$, $\sA _6$. 
Here, note that $S$ has no singular point of type $\sA _n$ with $n \ge 7$ because $\rho (\wS ) = 8$ and $\rho (S) \ge 2$. 
Moreover, $S$ has a singular point, which is not of type $\sA _1$ by Theorem \ref{CPW}. 
\begin{lem}\label{lem(4-3)}
If $S$ is of type $(\sA _5)'$ or $(\sA _5+\sA _1)'$, then $\Ampc (S) = \Amp (S)$. 
\end{lem}
\begin{proof}
By assumption, $S$ has a singular point $P$ of type $\sA _5$. 
Let $\wD _1+\dots + \wD _5$ be the connected component of $\wD$ corresponding to $P$ such that the weighted dual graph of $\wD _1+\dots + \wD _5$ is given as follows: 
\begin{align*}
\xygraph{\circ ([]!{+(0,.25)} {^{\wD_1}}) -[r] \circ ([]!{+(0,.25)} {^{\wD_2}}) -[r] \circ ([]!{+(0,.25)} {^{\wD_3}}) -[r] \circ ([]!{+(0,.25)} {^{\wD_4}})-[r] \circ ([]!{+(0,.25)} {^{\wD_5}})}
\end{align*}
By assumption, there exists a $(-1)$-curve $\wE _3$ on $\wS$ such that $(\wD _i \cdot \wE _3) = \delta _{i,3}$ for $i=1,\dots ,5$. 
Then a divisor $\wF := \wD_2+2(\wD_3+\wE_3)+\wD_4$ defines a $\bP ^1$-fibration $g := \Phi _{|\wF|}: \wS \to \bP ^1_{\bk}$, 
$\wD_1$ and $\wD_5$ become sections of $g$ and each component of $\wD-(\wD_1+\wD_5)$ is a fiber component of $g$. 
Note that the configuration of $\wF + \wD _1 + \wD _5$ looks like that in Figure \ref{fig(4-2)} (b). 
Moreover, $g$ satisfies $(\ast \ast )$, $(s,t)=(0,1)$ and $\gamma = 3$ provided that we use notations from \S \S \ref{4-1}. 
Hence, we obtain that $\Amp (S) = \Ampc (S)$ by Proposition \ref{prop(3-3)}. 
\end{proof}
\begin{lem}\label{lem(4-4)}
If $S$ is of type $\sA_3+\sA_2+\sA_1$, $\sA_3+3\sA_1$, $(\sA _3+\sA _1)'$ or $(\sA _3+2\sA _1)'$, then $\Ampc (S) = \Amp (S)$. 
\end{lem}
\begin{proof}
By assumption, $S$ has two singular points $P_1$ and $P_2$ of types $\sA _3$ and $\sA _1$, respectively. 
Let $\wD _1+\wD _2 + \wD _3$ (resp. $\wD _4$) be the connected component of $\wD$ corresponding to $P_1$ (resp. $P_2$) such that the weighted dual graph of $(\wD _1+\wD _2 + \wD _3) + \wD_4$ is given as follows: 
\begin{align*}
\xygraph{\circ ([]!{+(0,.25)} {^{\wD_1}}) -[r] \circ ([]!{+(0,.25)} {^{\wD_2}}) -[r] \circ ([]!{+(0,.25)} {^{\wD_3}}) [r] \circ ([]!{+(0,.25)} {^{\wD_4}})}
\end{align*}
By assumption and {\cite[pp.\ 21--22]{PW10}}, we may assume that there exists a $(-1)$-curve $\wE _2$ on $\wS$ such that $(\wD _i \cdot \wE _2) = \delta _{i,2} + \delta _{i,4}$ for $i=1,\dots ,4$. 
Then a divisor $\wF := \wD_2+2\wE_2+\wD_4$ defines a $\bP ^1$-fibration $g := \Phi _{|\wF|}: \wS \to \bP ^1_{\bk}$, 
$\wD_1$ and $\wD_3$ become sections of $g$ and each component of $\wD-(\wD_1+\wD_3)$ is a fiber component of $g$. 
Note that the configuration of $\wF + \wD _1 + \wD _3$ looks like that in Figure \ref{fig(4-2)} (c). 
Moreover, $g$ satisfies $(\ast \ast )$, $(s,t)=(0,1)$ and $\gamma = 2$ provided that we use notations from \S \S \ref{4-1}. 
Hence, we obtain that $\Amp (S) = \Ampc (S)$ by Proposition \ref{prop(3-4)}. 
\end{proof}
\begin{lem}\label{lem(4-5)}
If $S$ has a singular point of type $\sA _3$, $\sA _4$, $\sA _5$ or $\sA _6$, then $\Ampc (S) = \Amp (S)$.
\end{lem}
\begin{proof}
If $S$ is of type $(\sA _5)'$, $(\sA _5+\sA _1)'$, $\sA_3+\sA_2+\sA_1$, $\sA_3+3\sA_1$, $(\sA _3+\sA _1)'$ or $(\sA _3+2\sA _1)'$, then we know $\Ampc (S) = \Amp (S)$ by Lemmas \ref{lem(4-3)} and \ref{lem(4-4)}. 
Hence, we assume that $S$ is not of the above types in what follows. 
By assumption, $S$ has a singular point $P$ of type $\sA _n$ for some $n=3,4,5,6$. 
Let $\wD _1+\dots + \wD _n$ be the connected component of $\wD$ corresponding to $P$ such that the weighted dual graph of $\wD _1+\dots + \wD _n$ is given as follows: 
\begin{align*}
\xygraph{\circ ([]!{+(0,.25)} {^{\wD_1}}) -[r] \circ ([]!{+(0,.25)} {^{\wD_2}}) -[r] \cdots  -[r] \circ ([]!{+(0,.25)} {^{\wD_n}})}
\end{align*}

We first consider the case $n=3$. 
Since $S$ is not of type listed in Lemma \ref{lem(4-4)}, there exist two distinct $(-1)$-curves $\wE _2$ and $\wE _2'$ on $\wS$ such that $(\wE _2 \cdot \wE _2') = 0$ and $(\wD _i \cdot \wE _2) = (\wD _i \cdot \wE _2') = \delta _{i,2}$ for $i=1,2,3$ by Proposition \ref{prop(3-0)}. 
Then a divisor $\wF := \wE _2+\wD_2+\wE_2'$ defines a $\bP ^1$-fibration $g := \Phi _{|\wF|}: \wS \to \bP ^1_{\bk}$, 
$\wD_1$ and $\wD_3$ become sections of $g$ and each component of $\wD-(\wD_1+\wD_3)$ is a fiber component of $g$. 
Note that the configuration of $\wF + \wD _1 + \wD _3$ looks like that in Figure \ref{fig(4-2)} (d). 
Moreover, $g$ satisfies $(\ast \ast )$, $(s,t)=(1,0)$, $\beta ' = 1$ and $\beta = 1$ provided that we use notations from \S \S \ref{4-1}. 
Hence, we obtain that $\Amp (S) = \Ampc (S)$ by Proposition \ref{prop(3-6)} (1). 

In what follows, we consider the remaining case. 
Since $S$ is not of type $(\sA _5)'$ nor $(\sA _5+\sA _1)'$, there exist two distinct $(-1)$-curves $\wE _2$ and $\wE _{n-1}$ on $\wS$ such that $(\wE _2 \cdot \wE _{n-1}) = 0$ and$(\wD _i \cdot \wE _j) = \delta _{i,j}$ for $i=1,\dots ,n$ and $j=2,n-1$ by Proposition \ref{prop(3-0)}.  
Then a divisor $\wF := \wE _2 + \wD_2 + \dots + \wD _{n-1} + \wE _{n-1}$ defines a $\bP ^1$-fibration $g := \Phi _{|\wF|}: \wS \to \bP ^1_{\bk}$, 
$\wD_1$ and $\wD_n$ become sections of $g$ and each component of $\wD-(\wD_1+\wD_n)$ is a fiber component of $g$. 
Note that the configuration of $\wF + \wD _1 + \wD _n$ looks like that in Figure \ref{fig(4-2)} (e). 
Moreover, $g$ satisfies $(\ast \ast )$, $(s,t)=(1,0)$, $\beta ' = 1$ and $\beta \ge 2$ provided that we use notations from \S \S \ref{4-1}. 
Hence, we obtain that $\Amp (S) = \Ampc (S)$ by Proposition \ref{prop(3-6)} (2). 
\end{proof}
\begin{lem}\label{lem(4-6)}
If $S$ has a singular point of type $\sA _2$, then $\Ampc (S) = \Amp (S)$.
\end{lem}
\begin{proof}
Let $P$ be a singular point on $S$ of type $\sA _2$ and let $\wD _1+\wD _2$ be the connected component of $\wD$ corresponding to $P$ such that the weighted dual graph of $\wD _1+\wD _2$ is given as follows: 
\begin{align*}
\xygraph{\circ ([]!{+(0,.25)} {^{\wD_1}}) -[r] \circ ([]!{+(0,.25)} {^{\wD_2}})}
\end{align*}
Consider the divisor $\wDelta := -K_{\wS} - (\wD _1 + \wD _2)$ on $\wS$. 
By the Riemann-Roch theorem, we have $\dim |\wDelta| \ge 1$. 
Moreover, $(\wDelta )^2 = 0$ and $(\wDelta \cdot -K_{\wS}) = 2$. 
Since $\wS$ is a weak del Pezzo surface, we know that there exists a $0$-curve $\wC$ included in $|\wDelta|$ (see also {\cite[Lemma 5.2]{Saw24}}). 
Hence, $\wC$ defines a $\bP ^1$-fibration $g := \Phi _{|\wC|}: \wS \to \bP ^1_{\bk}$, 
$\wD_1$ and $\wD_2$ become sections of $g$ and each component of $\wD-(\wD_1+\wD_2)$ is a fiber component of $g$. 
Note that the configuration of $\wD _1 + \wD _2$ looks like that in Figure \ref{fig(4-2)} (f). 
Moreover, $g$ satisfies $(\ast \ast )$ and $(s,t)=(0,0)$ provided that we use notations from \S \S \ref{4-1}. 
Hence, we obtain that $\Amp (S) = \Ampc (S)$ by Proposition \ref{prop(3-5)}. 
\end{proof}
Hence, in every case, we obtain $\Amp (S) = \Ampc (S)$ by Lemmas \ref{lem(4-1)}, \ref{lem(4-2)}, \ref{lem(4-3)}, \ref{lem(4-4)}, \ref{lem(4-5)} and \ref{lem(4-6)}. 
The proof of Theorem \ref{main(1)} is thus completed. 
\begin{table}[t]
\begin{center}
\begin{tabular}{|c||c|c|c|}
\hline
Dynkin type & Lemma & $(r,s,t)$ &  Configuration of singular fibers \\
\hline \hline 

$\sE_6$ & Lemma \ref{lem(4-1)} &$(1,0,1)$& $\alpha_1 = 1$, $\gamma_1 = 5$ \\ \hline 
$\sD_6$ & Lemma \ref{lem(4-1)} &$(0,1,1)$&$(\beta_1,\beta_1')=(1,1)$, $\gamma_1=4$ \\ \hline 
$\sD_5+\sA_1$ & Lemma \ref{lem(4-2)} &$(1,0,1)$& $\alpha_1=2$, $\gamma_1 = 4$ \\ \hline 
$\sD_5$ & Lemma \ref{lem(4-2)} &$(2,0,1)$& $\{\alpha_i\}_{i=1,2}=\{(1)_2\}$, $\gamma_1 = 4$ \\ \hline 
$\sD_4+2\sA_1$& Lemma \ref{lem(4-1)} &$(0,1,2)$& $(\beta_j,\beta_j')=(1,1)$, $\{\gamma_{\ell}\}_{\ell=1,2} = \{(2)_2\}$ \\ \hline
$\sD_4+\sA_1$ & Lemma \ref{lem(4-1)} &$(0,2,1)$& $\{(\beta_j,\beta_j')\}_{j=1,2}=\{(1,1)_2\}$, $\gamma_1 = 2$ \\ \hline 
$\sD_4$ & Lemma \ref{lem(4-1)} &$(0,3,0)$& $\{(\beta_j,\beta_j')\}_{j=1,2,3}=\{(1,1)_3\}$ \\ \hline 

$\sA_6$ & Lemma \ref{lem(4-5)} &$(1,1,0)$& $\alpha_1=1$, $(\beta_1,\beta_1')=(4,1)$ \\ \hline 

$(\sA_5+\sA_1)'$& Lemma \ref{lem(4-3)} &$(2,0,1)$&$\{\alpha_i\}_{i=1,2} = \{ 1,2\}$, $\gamma_1=3$ \\ \hline 
$(\sA_5)'$ & Lemma \ref{lem(4-3)} &$(3,0,1)$&$\{\alpha_i\}_{i=1,2,3} = \{ (1)_3\}$, $\gamma_1=3$ \\ \hline 
$(\sA_5+\sA_1)''$& Lemma \ref{lem(4-5)} &$(1,1,0)$&$\alpha_1 = 2$, $(\beta_1,\beta_1')=(3,1)$ \\ \hline 
$(\sA_5)''$ & Lemma \ref{lem(4-5)} &$(2,1,0)$&$\{\alpha_i\}_{i=1,2} = \{ (1)_2\}$, $(\beta_1,\beta_1')=(3,1)$ \\ \hline 

\multirow{2}{*}{$\sA_4+\sA_2$} &Lemma \ref{lem(4-5)} &$(1,1,0)$&$\alpha_1=3$, $(\beta_1,\beta_1')=(2,1)$ \\ \cline{2-4}
& Lemma \ref{lem(4-6)} &$(2,0,0)$&$\{\alpha_i\}_{i=1,2} = \{ 1,5\}$ \\ \hline 
$\sA_4+\sA_1$ & Lemma \ref{lem(4-5)} &$(2,1,0)$&$\{\alpha_i\}_{i=1,2} = \{ 1,2\}$, $(\beta_1,\beta_1')=(2,1)$ \\ \hline 
$\sA_4$ & Lemma \ref{lem(4-5)} &$(3,1,0)$&$\{\alpha_i\}_{i=1,2,3} = \{ (1)_3\}$, $(\beta_1,\beta_1')=(2,1)$ \\ \hline 

$2\sA_3$ & Lemma \ref{lem(4-5)} &$(1,1,0)$&$\alpha_1 = 4$, $(\beta_1,\beta_1')=(1,1)$ \\ \hline 
\multirow{2}{*}{$\sA_3+\sA_2+\sA_1$} &Lemma \ref{lem(4-4)} &$(2,0,1)$&$\{\alpha_i\}_{i=1,2} = \{ 1,3\}$, $\gamma_1=2$ \\ \cline{2-4}
& Lemma \ref{lem(4-6)} &$(2,0,0)$&$\{\alpha_i\}_{i=1,2} = \{ 2,4\}$ \\ \hline 
$\sA_3+3\sA_1$ & Lemma \ref{lem(4-4)} &$(2,0,1)$&$\{\alpha_i\}_{i=1,2} = \{ (2)_2\}$, $\gamma_1=2$ \\ \hline 
\multirow{2}{*}{$\sA_3+\sA_2$} & Lemma \ref{lem(4-5)} &$(2,1,0)$&$\{\alpha_i\}_{i=1,2} = \{ 1,3\}$, $(\beta_1,\beta_1')=(1,1)$ \\ \cline{2-4}
& Lemma \ref{lem(4-6)} &$(3,0,0)$&$\{\alpha_i\}_{i=1,2,3} = \{ (1)_2,4\}$ \\ \hline 
$(\sA_3+2\sA_1)'$ & Lemma \ref{lem(4-4)} &$(3,0,1)$&$\{\alpha_i\}_{i=1,2,3} = \{ (1)_2,2\}$, $\gamma_1=2$ \\ \hline 
$(\sA_3+\sA_1)'$ & Lemma \ref{lem(4-4)} &$(4,0,1)$&$\{\alpha_i\}_{i=1,2,3,4} = \{ (1)_4\}$, $\gamma_1=2$ \\ \hline 
$(\sA_3+2\sA_1)''$ & Lemma \ref{lem(4-5)} &$(2,1,0)$&$\{\alpha_i\}_{i=1,2} = \{ (2)_2\}$, $(\beta_1,\beta_1')=(1,1)$ \\ \hline 
$(\sA_3+\sA_1)''$ & Lemma \ref{lem(4-5)} &$(3,1,0)$&$\{\alpha_i\}_{i=1,2,3} = \{ (1)_2,2\}$, $(\beta_1,\beta_1')=(1,1)$ \\ \hline 
$\sA_3$ & Lemma \ref{lem(4-5)} &$(4,1,0)$&$\{\alpha_i\}_{i=1,2,3,4} = \{ (1)_4\}$, $(\beta_1,\beta_1')=(1,1)$ \\ \hline 

$3\sA_2$ & Lemma \ref{lem(4-6)} &$(2,0,0)$&$\{\alpha_i\}_{i=1,2} = \{ (3)_2\}$ \\ \hline 
$2\sA_2+\sA_1$ & Lemma \ref{lem(4-6)} &$(3,0,0)$&$\{\alpha_i\}_{i=1,2,3} = \{ 1,2,3\}$ \\ \hline 
$\sA_2+3\sA_1$ & Lemma \ref{lem(4-6)} &$(3,0,0)$&$\{\alpha_i\}_{i=1,2,3} = \{ (2)_3\}$ \\ \hline 
$2\sA_2$ & Lemma \ref{lem(4-6)} &$(4,0,0)$&$\{\alpha_i\}_{i=1,2,3,4} = \{ (1)_3,3\}$ \\ \hline 
$\sA_2+2\sA_1$ & Lemma \ref{lem(4-6)} &$(4,0,0)$&$\{\alpha_i\}_{i=1,2,3,4} = \{ (1)_2,(2)_2\}$ \\ \hline 
$\sA_2+\sA_1$ & Lemma \ref{lem(4-6)} &$(5,0,0)$&$\{\alpha_i\}_{i=1,2,3,4,5} = \{ (1)_4,2\}$ \\ \hline 
$\sA_2$ & Lemma \ref{lem(4-6)} &$(6,0,0)$&$\{\alpha_i\}_{i=1,2,3,4,5,6} = \{ (1)_6\}$ \\ \hline 
\end{tabular}
\caption{The list of the configuration of $\bP^1$-fibration $g:\wS \to \bP^1_{\bk}$. }\label{table}
\end{center}
\end{table}
\smallskip

For the reader's convenience, we summarize in Table \ref{table} the configurations of the $\bP^1$-fibrations considered for each Dynkin type. 
We shall explain the notation of this table. 
``Dynkin type'' means the Dynkin type of $S$. 
``Lemma'' indicates the lemma in whose proof the corresponding $\bP^1$-fibration is constructed. 
``$(r,s,t)$'' denotes the triplet of numbers of singular fibers of $g:\wS \to \bP^1_{\bk}$ corresponding to (\ref{I-1}), (\ref{I-2}), and (\ref{II}). 
``Configurations of singular fibers'' means the types of singular fibers of $g:\wS \to \bP^1_{\bk}$ (see \S\S \ref{4-1} for the notation $\alpha_i$, $(\beta_j,\beta_j')$, and $\gamma_{\ell}$). 
Here, $(m)_k$ (resp. $(m_1,m_2)_k$) in this table denotes the repetition of a positive integer $m$ (resp. a pair of positive integers $(m_1,m_2)$) exactly $k$ times. 
Notice that if the Dynkin type of $\wS$ is $\sA_4+\sA_2$, $\sA_3+\sA_2+\sA_1$ or $\sA_3+\sA_2$, then there are two $\bP^1$-fibrations giving rise to $H$-polar cylinders for every ample $\bQ$-divisor $H$. 
As an example, we consider the case ${\rm Dyn}(S) = (\sA_3+2\sA_1)''$. 
Then the configuration of the $\bP^1$-fibration $g$ looks like that in Figure \ref{(A_3+2A_1)''}. 

\begin{figure}[t]
\begin{center}\scalebox{0.5}{\begin{tikzpicture}
\draw [thick] (-0.5,0.5)--(8.5,0.5);
\node at (-0.9,0.5) {\Large $\wD _0$};
\draw [thick] (-0.5,4.5)--(8.5,4.5);
\node at (-0.9,4.5) {\Large $\wD _{\infty}$};

\draw [dashed,thick] (1,0)--(0,2);
\draw [thick] (0.2,1.25)--(0.2,3.75);
\draw [dashed,thick] (0,3)--(1,5);
\node at (1,-0.4) {\Large $\wE _1'$};
\node at (0.8,2.5) {\Large $\wD _{1,1}$};
\node at (1,5.4) {\Large $\wE _1$};

\draw [dashed,thick] (4,0)--(3,2);
\draw [thick] (3.2,1.25)--(3.2,3.75);
\draw [dashed,thick] (3,3)--(4,5);
\node at (4,-0.4) {\Large $\wE _2'$};
\node at (3.8,2.5) {\Large $\wD _{2,1}$};
\node at (4,5.4) {\Large $\wE _2$};

\draw [thick] (6.5,0)--(6.5,5);
\draw [dashed,thick] (5.5,1.75)--(7.5,1.75);
\draw [dashed,thick] (5.5,3.25)--(7.5,3.25);

\node at (6.5,-0.4) {\Large $\wD_{3,0}$};
\node at (8,1.75) {\Large $\wE _3'$};
\node at (8,3.25) {\Large $\wE _3$};
\end{tikzpicture}}\end{center}
\caption{The configuration of $\wS$ with ${\rm Dyn} (S) = (A_3+2A_1)''$}\label{(A_3+2A_1)''}
\end{figure}
\section*{Data Availability}
Data sharing not applicable to this article as no datasets were generated or analysed during the current study. 
\section*{Conflict of Interest}
The author declares no competing interests. 


\begin{thebibliography}{99}
\bibitem{Bel17} G. Belousov, {\em Cylinders in del Pezzo surfaces with du Val singularities}, Bull. Korean Math. Soc. {\bf 54} (2017), 1655--1667. 
\bibitem{Bel23} G. Belousov, {\em Cylinders in del Pezzo surfaces of degree two}, In: {\em Birational Geometry, K\"{a}hler-Einstein
Metrics and Degenerations}, Springer Proc. Math. Stat., Vol. 409, Springer, Cham, 2023, 17--70. 
\bibitem{BW79} J. W. Bruce and C. T. C. Wall, {\em On the classification of cubic surfaces}, J. Lond. Math. Soc. {\bf 19} (1979), 245--256. 
\bibitem{Che21} I. Cheltsov, {\em Cylinders in rational surfaces}, Sb. Math. {\bf 212} (2021), 399--415. 
\bibitem{CDP18} I. Cheltsov, A. Dubouloz and J. Park, {\em Super-rigid affine Fano varieties}, Compos. Math. {\bf 154} (2018), 2462--2484. 
\bibitem{CPPZ21} I. Cheltsov, J. Park, Y. Prokhorov and M. Zaidenberg, {\em Cylinders in Fano varieties}, EMS Surv. Math. Sci. {\bf 8} (2021), 39--105.
\bibitem{CPW16a} I. Cheltsov, J. Park and J. Won, {\em Affine cones over smooth cubic surfaces}, J. Eur. Math. Soc. {\bf 18} (2016), 1537--1564. 
\bibitem{CPW16b} I. Cheltsov, J. Park and J. Won, {\em Cylinders in singular del Pezzo surfaces}, Compos. Math. {\bf 152} (2016), 1198--1224. 
\bibitem{CPW17} I. Cheltsov, J. Park and J. Won, {\em Cylinders in del Pezzo surfaces}, Int. Math. Res. Not. {\bf 2017} (2017), 1179--1230. 
\bibitem{CD} M. Chitayat and A. Dubouloz, {\em The rigid Pham-Brieskorn threefolds}, Adv. Math. {\bf 482} (2025), 110649, 33pp. 
\bibitem{CT88} D. F. Coray and M. A. Tsfasman, {\em Arithmetic on singular Del Pezzo surfaces}, Proc. Lond. Math. Soc. (3) {\bf 57} (1988), 25--87. 
\bibitem{Dur79} A. H. Durfee, {\em Fifteen characterizations of rational double points and simple critical points}, Enseign. Math. {\bf 25} (1979), 131--163. 
\bibitem{Dol12} I. V. Dolgachev, {\em Classical Algebraic Geometry: a modern view}, Cambridge Univ. Press, Cambridge, 2012. 
\bibitem{KKW24} I.-K. Kim, J. Kim and J. Won, {\em $K$-unstable singular del Pezzo surfaces without anticanonical polar cylinder},  Int. Math. Res. Not. {\bf 2024} (2024), 12599--12619. 
\bibitem{KKW} I.-K. Kim, J. Kim and J. Won, {\em Rigid affine cones over singular del Pezzo surfaces}, preprint, \href{https://arxiv.org/abs/2506.01310}{arXiv:2506.01310}, 2025. 
\bibitem{KP21} J. Kim and J. Park, {\em Generic flexibility of affine cones over del Pezzo surfaces of degree $2$}, Int. J. Math. {\bf 32} (2021), Article ID 2150104, 18pp. 
\bibitem{KW25} J. Kim and J. Won, {\em Cylinders in smooth del Pezzo surfaces of degree $2$}, Adv. Geom. {\bf 25} (2025), 71--91. 
\bibitem{KPZ11} T. Kishimoto, Y. Prokhorov and M. Zaidenberg, {\em Group actions on affine cones}, In: {\em Affine Algebraic Geometry}, CRM Proc. Lecture Notes, Vol. 54, American Mathematical Society, Providence, RI, 2011, 123--163. 
\bibitem{KPZ13} T. Kishimoto, Y. Prokhorov and M. Zaidenberg, {\em $\mathbb{G}_a$-actions on affine cones}, Transform. Groups {\bf 18} (2013), 1137--1153. 
\bibitem{KPZ14} T. Kishimoto, Y. Prokhorov and M. Zaidenberg, {\em Unipotent group actions on del Pezzo cones}, Algebr. Geom. {\bf 1} (2014), 46--56. 
\bibitem{Koj02} H. Kojima, {\em Algebraic compactifications of some affine surfaces}, Algebra Colloq. {\bf 9} (2002), 417--425. 
\bibitem{MW18} L. Marquand and J. Won, {\em Cylinders in rational surfaces}, Eur. J. Math. {\bf 4} (2018), 1161--1196. 
\bibitem{Par22} J. Park, {\em $\mathbb{G}_a$-actions on the complements of hypersurfaces}, Transform. Groups {\bf 27} (2022), 651--657. 
\bibitem{PW10} J. Park and J. Won, {\em  Log canonical thresholds on del Pezzo surfaces of degree $\ge 2$}, Nagoya Math. J. {\bf 200} (2010), 1--26. 
\bibitem{PW16} J. Park and J. Won, {\em Flexible affine cones over del Pezzo surfaces of degree $4$}, Eur. J. Math. {\bf 2} (2016), 304--318. 
\bibitem{Pre13}  A. Y. Perepechko, {\em Flexibility of affine cones over del Pezzo surfaces of degree 4 and 5}, Funct. Anal. Appl. {\bf 47} (2013), 284--289. 
\bibitem{Saw24} M. Sawahara, {\em Cylinders in canonical del Pezzo fibrations}, Ann. Inst. Fourier {\bf 74} (2024), 1--69. 
\bibitem{Saw25} M. Sawahara, {\em Cylindrical ample divisors on Du Val del Pezzo surfaces}, Forum Math. {\bf 37} (2025), 1597--1619. 
\bibitem{Ura81} T. Urabe, {\em On singularities on degenerate del Pezzo surfaces of degree 1, 2}, In: {\em Singularities, Part 2 (Arcata, Calif., 1981)}, Proc. Sympos. Pure Math., Vol. 40, American Mathematical Society, Providence, RI, 1983, 587--591. 
\bibitem{Zha88} D.-Q. Zhang, {\em Logarithmic del Pezzo surfaces of rank one with contractible boundaries}, Osaka J. Math. {\bf 25} (1988) 461--497. 
\end{thebibliography}
\end{document}